\documentclass[12pt,a4paper]{article}

\usepackage{amsmath,amstext,amssymb,amscd,euscript}
 \usepackage{graphicx}

\usepackage[russian,english]{babel}
\usepackage[cp1251]{inputenc}

\newcommand{\Xcomment}[1]{}
\newtheorem{theorem}{Theorem}[section]
\newtheorem{lemma}[theorem]{Lemma}
\newtheorem{corollary}[theorem]{Corollary}

\newenvironment{proof}{\noindent{\bf Proof}\/}%
{\hfill$\qed$\medskip}

\def\qed{\Box}

\makeatletter \@addtoreset{equation}{section} \makeatother

\newenvironment{numitem1}{\refstepcounter{equation}\begin{enumerate}%
\item[(\thesection.\arabic{equation})]}{\end{enumerate}}

\newcommand{\refeq}[1]{(\ref{eq:#1})}  % reference to equation

 \makeatletter
\renewcommand{\section}{\@startsection{section}{1}{0pt}%
{-3.5ex plus -1ex minus -.2ex}{2.3ex plus .2ex}%
{\normalfont\Large}}
 \makeatother

 \makeatletter
\renewcommand{\subsection}{\@startsection{subsection}{2}{0pt}%
{-3.0ex plus -1ex minus -.2ex}{-1.5ex plus .2ex}%
{\normalfont\normalsize\bf}}
 \makeatother

  \Xcomment{
 \makeatletter
\renewcommand{\subsection}{\@startsection{subsection}{2}{0pt}%
{-3.0ex plus -1ex minus -.2ex}{1.5ex plus .2ex}%
{\normalfont\normalsize\bf}}
 \makeatother
   }

 \newcommand{\SEC}[1]{\ref{sec:#1}}  % reference to section
\newcommand{\SSEC}[1]{\ref{ssec:#1}}  % reference to subsection
\def\Rset{{\mathbb R}}

\def\Zset{{\mathbb Z}}

\def\Ascr{{\cal A}}
\def\Bscr{{\cal B}}
\def\Cscr{{\cal C}}

\def\Escr{{\cal E}}
\def\Fscr{{\cal F}}

\def\Hscr{{\cal H}}

\def\Kscr{{\cal K}}
\def\Lscr{{\cal L}}
\def\Mscr{\EuScript{M}}
\def\Oscr{{\cal O}}

\def\Rscr{{\cal R}}
\def\Sscr{{\cal S}}
\def\Tscr{{\cal T}}
\def\Uscr{{\cal U}}

\def\frakC{\mathfrak{C}}

\def\frakS{\mathfrak{S}}
\def\frakL{\mathfrak{L}}
\def\frakU{\mathfrak{U}}

\def\frakO{\mathfrak{O}}

\def\tilde{\widetilde}
\def\hat{\widehat}
\def\bar{\overline}
\def\eps{\varepsilon}

\def\symm{\lozenge}

\def\xmin{x^{\rm min}}
\def\xmax{x^{\rm max}}

\def\Sscrpr{\Sscr^{\rm pr}}
\def\precpr{\prec^{\rm pr}}
\def\phipr{\phi^{\rm pr}}

\def\bmax{b^{\rm max}}

\def\Deltain{\Delta^{\rm in}}
\def\xin{x^{\rm in}}

\def\Deltaout{\Delta^{\rm out}}
\def\xout{x^{\rm out}}
\def\Kin{K^{\rm in}}
\def\Kout{K^{\rm out}}

\def\onebf{{\bf 1}}

\begin{document}
\parskip=2pt

\title{On a stable partnership problem with integer choice functions}

\author{Alexander V.~Karzanov
\thanks{Central Institute of Economics and Mathematics of
the RAS, 47, Nakhimovskii Prospect, 117418 Moscow, Russia; email:
akarzanov7@gmail.com.}
}
\date{}

 \maketitle
\vspace{-0.7cm}

\begin{abstract}
We consider a far generalization of the well-known stable roommates and non-bipartite stable allocation problems. In its setting, one is given a finite non-bipartite graph $G=(V,E)$ with nonnegative integer edge \emph{capacities} $b(e)\in\Zset_+$, $e\in E$, in which for each vertex (``agent'') $v\in V$, the preferences on the set $E_v$ of its incident edges are given via a \emph{choice function} $C_v$ acting on the vectors in $\Zset_+^{E_v}$ bounded by the capacities and obeying the standard axioms of \emph{substitutability} and \emph{size monotonicity}. We refer to the related stability problem as the \emph{stable partnership problem with integer choice functions}, or \emph{SPPIC} for short.

Extending well-known results for particular cases, we give a solvability criterion for SPPIC and develop an algorithm of finding a stable solution, called a \emph{stable partnership}, or establishing that there is none. Moreover, in general the algorithm constructs a pair $(x,\Kscr)$ such that $x\in \Zset_+^E$ and $\Kscr$ is a set of pairwise edge-disjoint odd cycles in $G$ satisfying the following properties: if $\Kscr=\emptyset$, then $x$ is a stable partnership, whereas if $\Kscr$ is nonempty, then a stable partnership does not exist, and in this case, the set $\Kscr$ is determined canonically. 

Our constructions essentially use earlier author's results on the corresponding bipartite counterpart of SPPIC.
 \medskip

 \end{abstract}

%----------------------- Sec. 1

\section{Introduction}  \label{sec:intr}

Starting with the famous work of Gale and Shapley on stable marriages~\cite{GS}, there have appeared a big amount of valuable researches on stable assignments in finite bipartite graphs $G=(V,E)$. One can conditionally distinguish between ``boolean'' and ``numerical'' cases, where in the former (latter) case, an assignment is meant to be a subset of edges (resp. a real- or integer-valued function on $E$). One of the most popular ``numerical'' stability models is the \emph{stable allocation} problem due to Baiou and Balinsky~\cite{BB} in which (in the real version) the edges $e\in E$ are endowed with nonnegative upper bounds, or \emph{capacities} $b(e)\in\Rset_+$, the vertices $v\in V$ are endowed with \emph{quotas} $q(v)\in \Rset_+$, and the preferences of each vertex (``agent'') $v$ are given via a prescribed linear order on the set $E_v$ of edges incident to $v$.

As a far generalization of the stable allocation and other known problems (of ``one-to-many'', ``many-to-many'', ``diversifying'' types, etc.), Alkan and Gale~\cite{AG} introduced and studied a general stability model on a bipartite graph $G=(V,E)$ in which preferences at each vertex $v\in V$ depend on a \emph{choice function} (CF) $C_v$ that acts on the vectors in a closed subset $\Bscr_v$ of the box $\{z\in\Rset_+^{E_v}\colon z(e)\le b(e)\;\forall e\in E_v\}$. (The entire box and its integer sublattice are typical samples of $\Bscr_v$.) It is assumed that each $C_v$ obeys standard axioms of \emph{consistence} and \emph{substitutability} (going back to Kelso and Crawford~\cite{KC}, Roth~\cite{roth} and Blair~\cite{blair}). Central results in~\cite{AG} are that under these conditions a stable assignment (called ``stable schedule matching'' there) always exists and the set of these forms a lattice $\Lscr$. Moreover, by adding one more axiom, of \emph{size monotonicity}, the lattice $\Lscr$ becomes distributive. 

Recently an \emph{integer} version of Alkan--Gale's model was studied in-depth in~\cite{karz}, where each $C_v$ acts on the integer box $\Bscr_v=\{z\in\Zset_+^{E_v}\colon z(e)\le b(e)\;\forall e\in E_v\}$; we will refer to this version as the \emph{integer Alkan--Gale's problem}, briefly \emph{IAGP}. As one of the main results in~\cite{karz}, an algorithm to construct a poset representation $(\Pi,\lessdot)$ of the lattice $\Lscr$ of stable assignments is developed (the existence of such a poset is prompted by the fact that $\Lscr$ is distributive.) More precisely, $\Pi$ is a family (admitting replications) of pairs $(R,\tau)$ formed by certain edge-simple cycles, or \emph{rotations}, $R$ in $G$ along with weights $\tau\in\Zset_{>0}$, and $\lessdot$ is a partial order on $\Pi$. They possess the property that $\Lscr$ is isomorphic to the lattice of closed functions for $(\Pi,\lessdot)$, where a function $\lambda:\Pi\to\Zset_+$ not exceeding $\tau$ is called \emph{closed} if for pairs $(R,\tau),(R',\tau')\in\Pi$ satisfying $(R,\tau)\lessdot(R',\tau')$ and $\lambda(R',\tau)>0$, there holds $\lambda(R,\tau)=\tau$. 

(Note also that, as is shown in~\cite{karz}, the size $|\Pi|$ does not exceed $\bmax|E|$, where $\bmax:=\max\{b(e)\colon e\in E\}$, and the poset can be constructed in pseudo polynomial time. On the other hand,~\cite{karz} gives an example with a small graph $G$ and large capacities for which $|\Pi|=\bmax$; in that example there are only two rotations, which alternate $\bmax$ times when moving from the minimal to maximal element of $\Lscr$.)

Concerning aspects of computational complexity, one assumes that each choice function $C_v$ is given by an \emph{oracle} that, being asked of a vector $z\in\Zset_+^{E_v}$ bounded by the capacities $b$, outputs the ``acceptable'' vector $C_v(z)$. As a rule, one assumes that each application of $C_v$ (``oracle call'')  takes time polynomial in $|E_v|$; moreover, sometimes it may be reasonable to estimate the number of oracle calls rather than the complexity of their implementations, like in a wide scope of problems dealing with oracles.

Next we come to the subject of our interest in this paper. Since mid-1980s, in parallel to a wide stream of researches devoted to bipartite stability models, there have appeared breakthrough results on their non-bipartite counterparts. As a rule, the latter problems look more intricate and the existence of a stable solution is not guaranteed; a simple counterexample for a non-bipartite analog of the classical stable marriage problem, the so-called \emph{stable roommates} one, briefly \emph{SRP}, was demonstrated already in Gale and Shapley~\cite{GS}. Irving~\cite{irv} presented a linear time algorithm that either finds a solution to SRP -- a \emph{stable matching} -- or declares that there is none. Tan~\cite{tan} established necessary and sufficient conditions for the existence of a stable matching in SRP, relying on a new structure introduced in that work, the so-called \emph{stable partition}. It was shown in~\cite{tan} that for an arbitrary graph $G=(V,E)$ and linear orders on the sets $E_v$, $v\in V$, a stable partition always exists and can be found in linear time. Roughly speaking, a stable partition  consists of two parts: a matching $M$ in $G$ and a set $\Kscr$ of odd cycles; this $\Kscr$ is determined canonically and serves as an ``obstacle'' to the solvability of SRP, in the sense that a stable matching exists, and $M$ is one of them, if and only if $\Kscr$ is empty. See also~\cite{TH} for additional results on stable partitions in SRP, and~\cite{hsu} for reducing to a stable marriage problem in a symmetric bipartite graph with symmetric linear orders.

Subsequently valuable results have been obtained for more general stability problems on non-bipartite graphs; we point out three works in this direction (of which two will be used in our further description). In \cite{BF} and~\cite[Sec.~5]{DM}, a non-bipartite version of the stable integer allocation problem is considered, and~\cite{flein} deals with a non-bipartite ``boolean'' stability problem in which the preferences depend on  substitutable and increasing (viz. size monotone) choice functions. Note that~\cite{BF,flein} completely characterize obstacles to the existence of stable solutions; again these are formed by collections  of pairwise edge disjoint simple (in~\cite{BF}) or edge-simple (in~\cite{flein}) odd cycles. 

This paper is devoted to a non-bipartite version of the above-mentioned integer Alkan--Gale's problem (IAGP). Here we are given a non-bipartite graph $G=(V,E)$ with capacities $b(e)\in\Zset_+$ of edges $e\in E$ in which each vertex $v\in V$ is endowed with a choice function $C_v$ mapping the integer box $\Bscr_v$ into itself obeying the axioms of substitutability and size monotonicity. We refer to the corresponding stability problem as the \emph{stable partnership problem with integer choice functions}, or \emph{SPPIC} for short. This generalizes the above-mentioned non-bipartite stability problems, and our main goal is to work out a reasonable method to find a stable partnership for $(G,b,C)$ as above, or to establish that there is none.

Our method is based on a reduction to the corresponding symmetric bipartite case, namely, to a symmetric version of IAGP (like it was done for earlier problems in~\cite{hsu,DM}). More precisely, we associate to $G=(V,E)$ the bipartite graph $G^\symm=(V^\symm,E^\symm)$ of which vertices are copies $v^0,v^1$ of vertices $v\in V$, and edges are copies $u^0v^1$ and $v^0u^1$ of edges $uv\in E$. The edge capacities $b^\symm$ and choice functions $C^\symm$ for $G^\symm$ are inherited from $b$ and $C$ in a natural way. The resulting structure $(G^\symm,b^\symm,C^\symm)$ possesses the property that the stable partnerships $x\in \Zset_+^E$ for $(G,b,C)$ (if exist) one-to-one correspond to the symmetric stable vectors $\tilde x\in E^\symm$ for $(G^\symm,b^\symm,C^\symm)$ (i.e. satisfying $\tilde x(u^0v^1)=\tilde x(v^0u^1)=x(uv)$ for all $uv\in E$). 

Relying on results in~\cite{karz}, we find the family $\Pi$ of all weighted rotations for $(G^\symm,b^\symm,C^\symm)$ and extract from $\Pi$ the set $\frakO$ of all \emph{self-symmetric} rotations $R$ whose weights $\tau_R$ are \emph{odd}. We show that these rotations are pairwise edge-disjoint, and slightly modifying a method from~\cite{karz}, construct a stable vector $\tilde x\in \Zset_+^{E^\symm}$ with the following additional properties: $\tilde x$ is symmetric outside $\frakO$, and $|\tilde x(u^0v^1)-\tilde x(v^0u^1)|=1$ for all symmetric edges $u^0v^1,v^0u^1$ contained in rotations in $\frakO$. Moreover, one shows that any stable $\tilde x$ which is symmetric outside $\frakO$ has different values for each pair of symmetric edges covered by $\frakO$ (so $\frakO$ plays a role of ``obstacle'', in a sense). Therefore, a symmetric solution for $(G^\symm,b^\symm,C^\symm)$ exists if and only if the set $\frakO$ (which is canonical) is empty.

Now taking the natural images of $\tilde x$ and $\frakO$ in the original graph $G$, we obtain a vector $x\in E$ and a set $\Kscr$ of pairwise edge-disjoint edge-simple odd cycles in $G$ such that: if $\Kscr=\emptyset$, then $x$ is a stable partnership, and conversely, if $\Kscr\ne\emptyset$, then no stable partnership can exist. Also, without explicitly appealing to the above reduction to IAGP, we are able to give an intrinsic definition of appropriate pairs $(x,\Kscr)$ for $(G,b,C)$, referring to them as stable \emph{half-partnership} (borrowing terminology from~\cite{flein}). On this way, generalizing the corresponding structural result for a ``boolean'' case with choice functions from~\cite{flein} (and earlier results of this sort, e.g. on stable partitions in~\cite{tan}), our main structural theorem for SPPIC (Theorem~\ref{tm:nonbipart}) says that a half-partnership $(x,\Kscr)$ for $(G,b,C)$ always exists, and that the part $\Kscr$ is the same for all half-partnerships.

Note also that the task of finding $(x,\Kscr)$ has in essence the same computational complexity as that of constructing the rotationsl poset $(\Pi,\lessdot)$ in IAGP.

This paper is organized as follows. Section~\SEC{defin} contains basic definitions and settings. Section~\SEC{back} gives a review of constructions and results from~\cite{karz} on the stability problem IAGP in a general bipartite graph. A reduction from the non-bipartite stability problem SPPIC in $(G,b,C)$ to a symmetric version of IAGP is described in Sect.~\SEC{symmetr}. Here we analyze in detail properties of the rotational poset $(\Pi,\lessdot)$ arising in the symmetric case, explain how the symmetric stability problem for $(G^\symm,b^\symm,C^\symm)$ is related to the so-called \emph{balanced} and \emph{quasi-balanced} closed functions in this poset, and establish key properties of such functions (in Theorem~\ref{tm:algQB}). In Sect.~\SEC{nonbipartite} we return to the initial non-bipartite $(G,b,C)$, give a definition of stable half-partnerships $(x,\Kscr)$ in SPPIC, and prove the main result on these objects (Theorem~\ref{tm:nonbipart}). The concluding Sect.~\SEC{concl} briefly outlines two additional topics: a special case of SPPIC arising by imposing one more axiom on the choice functions (following the special case of IAGP considered in~\cite[Sec.~7]{karz}), and a non-bipartite analog of a weakened version of the stable allocation problem studied in~\cite{karz2}.

%----------------------- Sec. 2

\section{Definitions and settings} \label{sec:defin}

In the stable partnership model of our interest, we deal with a finite graph $G=(V,E)$ in which the edges $e\in E$ are endowed with nonnegative integer \emph{capacities} $b(e)\in\Zset_{+}$.
One may assume, without loss of generality, that the graph $G$ is connected and
has no multiple edges. Then $|V|-1\le |E|\le \binom{|V|}{2}$. The
edge connecting vertices $u$ and $v$ may be denoted as $uv$.

For a vertex $v\in V$, the set of its incident edges is denoted by $E_v$. We write $\Bscr=\Bscr^{(b)}$ for the nonnegative \emph{integer} box $\{x\in \Zset_+^E\colon x\le b\}$. For $v\in V$, the restriction of $\Bscr$ to the set $E_v$ is denoted by $\Bscr_v$.
\medskip

\noindent$\bullet$ \textbf{Choice functions.} Each vertex (``agent'') $v\in V$ can prefer one vector in $\Bscr_v$ to another. The preferences depend on a \emph{choice function} (CF) related to $v$. This is a map $C=C_v$ of $\Bscr_v$ into itself satisfying $C(z)\le z$ for all $z\in \Bscr_v$. Moreover, $C$ obeys two well known axioms on pairs of vectors $z,z'\in\Bscr_v$. They require that:

 \begin{description}
 \item[\rm(SUB)] (\emph{substitutability}): if $z\ge z'$, then $C(z)\wedge z'\le C(z')$;
  \item[\rm(MON)] (\emph{size monotonicity}): if $z\ge z'$, then $|C(z)|\ge |C(z')|$.
  \end{description}

\noindent Hereinafter, for $a,b\in\Rset^S$, the functions $a\wedge b$ and $a \vee b$  take the values $\min\{a(e),b(e)\}$ and $\max\{a(e),b(e)\}$, $e\in S$, respectively. Also for a numerical function $a$ on a finite set $S$, we write $a(S)$ for $\sum(a(e) : e\in S)$, and $|a|$ for $\sum(|a(e)|\colon e\in S)$. In particular,
$a(S)=|a|$ if $a$ is nonnegative. 

(Regarding historical backgrounds, one can mention that (SUB) and (MON) imply the property of \emph{consistence} (CON): if $z\ge z'\ge C(z)$, then $C(z')=C(z)$. The latter implies that any $z\in\Bscr_v$ satisfies $C(C(z))=C(z)$. In turn, (SUB) and (CON) imply the \emph{stationarity}: any $z,z'\in\Bscr_v$ satisfy $C(z\vee z')=C(C(z)\vee z')$, which is analogous to the \emph{path independence} property introduced by Plott~\cite{plott} for boolean choice functions. A special case of~(MON) is the \emph{quota filling} condition (QF); it is applied when a vertex $v\in V$ is endowed with a \emph{quota} $q(v)\in \Zset_+$  and requires that  $|C(z)|=\min\{|z|,q(v)\}$.)
 \medskip

\noindent\textbf{Example 1.} 
A popular instance of choice functions $C_v$ obeying (SUB) and (QF) is generated by a \emph{linear order} $>_v$ on the set $E_v$. Here
for $e,e'\in E_v$ with $e>_v e'$, the ``agent'' $v$ is said to prefer the edge (``contract'') $e$ to $e'$. Then $C_v$ is defined by the following rule: for $z\in\Bscr_v$, if $|z|\le q(v)$, then $C_v(z):=z$; and if $|z|> q(v)$, then, renumbering the edges in $E_v$ as $e_1,\ldots,e_{|E_v|}$ so that $e_i>_v e_{i+1}$ for each $i$, take the maximal $j$ satisfying $r:=\sum(z(e_i)\colon i\le j)\le q(v)$ and define $C_v(z):=(z(e_1),\ldots,z(e_j),q(v)-r,0,\ldots,0)$. The CFs of this sort occur in the stable allocation problem by Baiou--Balinsky~\cite{BB}.
 \medskip

\noindent$\bullet$ \textbf{Stability.} For a vertex $v\in V$ and a function $x$ on $E$, let $x_v$ denote the restriction of $x$ to  $E_v$. A vector (function) $z\in \Bscr_v$ is called  \emph{acceptable} if $C_v(z)=z$; the collection of such vectors is
denoted by $\Ascr_v$. Extending this to the whole $\Bscr$, one says that $x\in\Bscr$ is (globally) acceptable if $x_v\in\Ascr_v$ for all $v\in V$. The collection of acceptable vectors in $\Bscr$ is denoted by $\Ascr$, and using terminology from~\cite{flein}, we refer to each of them as a vector of (acceptable) partnerships between ``agents'', or simply as a \emph{partnership}. 
For $v\in V$, the CF $C_v$ establishes preference relations on acceptable functions on $E_v$ as follows: $z\in\Ascr_v$ is regarded as \emph{preferred} to $z'\in\Ascr_v-\{z\}$ if
   \begin{equation} \label{eq:zzp}
   C_v(z\vee z')=z,
   \end{equation}
which is denoted as $z'\prec_v z$. The relation $\prec_v$ is shown to be transitive. 
 \medskip

\noindent\textbf{Definition 1.} Let $v\in V$ and $z\in\Ascr_v$. We say that an edge $e\in E_v$ is \emph{interesting} for $v$ under $z$ if there exists $z'\in\Bscr_v$ such that
  \begin{equation} \label{eq:inter_e}
  z'(e)>z(e),\quad \mbox{$z'(e')=z(e')$ for all $e'\ne e$,}\quad \mbox{and $C_v(z')(e)>z(e)$}
  \end{equation}
(the term ``interesting'' borrowed from~\cite{karz} is justified by a hope that an increase at $e$ could lead to a better assignment for $v$). Extending this to functions on $E$, we say that an edge $e=uv\in E$ is \emph{interesting} for a vertex (``agent'') $w\in\{u,v\}$ under a partnership $x\in\Ascr$ if so is for $w$ and $x_w$. If  $e=uv\in E$ is interesting under $x$ for both vertices $u$ and $v$, then the edge $e$ is called \emph{blocking} $x$. A partnership $x\in \Ascr$ is called \emph{stable} if no edge in $E$ blocks $x$. The set of stable partnerships is denoted by $\Sscr=\Sscr_{G,b,C}$.
\medskip

When the input graph $G$ is bipartite, a stable partnership exists for any capacities $b$ and choice functions $C_v$ as above, due to a general result on stability for two-sided markets by Alkan and Gale~\cite{AG}. As to non-bipartite graphs, a stable partnership need not exist even in simple cases, as noticed already in~\cite{GS}. To recognize solvable non-bipartite cases and, moreover, to establish additional useful properties, one can apply a reduction to a symmetric bipartite counterpart of $(G,b,C)$, by extending methods due to Hsueh~\cite{hsu} (for stable ordinary matchings) and Dean and Munshi~\cite[Sec.~5]{DM} (for stable allocations). This is performed as follows.

Associate to $G=(V,E)$ the bipartite graph $G^\symm$ in which the vertex set $V^\symm$ is formed by copies $v^0,v^1$ of vertices $v\in V$, and the edge set $E^\symm$ by copies $u^0v^1$ and $v^0u^1$ of edges $uv\in E$. The sets $V^i:=\{v^i\colon v\in V\}$, $i=0,1$, are the vertex parts (color classes) of $G^\symm$. Each edge $u^0v^1$ inherits the capacity $b(uv)=:b^\symm(u^0v^1)$. For each vertex $v^i\in V^\symm$, we denote by $E_{v^i}$ the set of edges of $G^\symm$ incident with $v^i$ and assign the choice function $C^\symm_{v^i}$ acting within $E_{v^i}$ to be the natural copy of $C_v$.

There is a natural bijection between the domain $\Bscr=\{x\in \Zset_+^E\colon x\le b\}$ and the set $\Bscr^{\symm}$ of \emph{symmetric} functions on $E^\symm$ bounded by $b^\symm$; namely, the bijection $x\stackrel{\beta}\longmapsto y$ given by $x(uv)=y(u^0v^1)=y(v^0u^1)$ for each $uv\in E$. One can see that
 \begin{numitem1} \label{eq:biject}
if $x\in\Ascr$ is stable for $(G,b,C)$, then the symmetric vector $\beta(x)$ is stable for $(G^\symm,b^\symm,C^\symm)$, and vice versa.
 \end{numitem1} 
 
Therefore, our stable partnership problem with a non-bipartite graph $G$ and $b,C$ as above is reduced to finding a stable symmetric solution to the corresponding stability problem with $G^\symm,b^\symm,C^\symm$, or establishing that there is none. 
In the next section we will give a review of constructions and results from~\cite{AG} and~\cite{karz} for the bipartite case that will be important for our reduction method.

%----------------------- Sec. 3

\section{Backgrounds in the bipartite case} \label{sec:back}

Throughout this section we deal with the stability model in the case when the graph $G=(V,E)$ is \emph{bipartite} and the edge capacities $b$ and choice functions $C_v$ ($v\in V$) are integer-valued as above. The vertices of $G$ are partitioned into two parts (or sides, color classes) $W$ and $F$, conditionally called the sets of \emph{workers} and \emph{firms}, respectively. 

As before, we denote the sets of acceptable and stable vectors (``partnerships'') $x\in\Zset_+^E$ for our model with $G,b,C$ by $\Ascr$ and $\Sscr$, respectively. Also for $v\in V$, we denote by $\prec_v$ the preference relation on vectors in $\Ascr_v$ defined as in~\refeq{zzp}. These preferences are extended, in a natural way, to the acceptable vectors on $E$. Here for distinct $x,y\in\Ascr$, we write $x\prec_F y$ if $x_v\preceq_v y_v$ holds for all ``firms'' $v\in F$. The preferences relative to ``workers'' are defined in a similar way and denoted as $\prec_W$.

Our model with $G,b,C$ as above is a special case of Alkan--Gale's general stability model on two-sided markets. We will utilize two important facts from their theory (generalizing a variety of well-known earlier results). Namely, the following is valid:
 \begin{numitem1} \label{eq:AG}
 \begin{itemize}
\item[(a)] (\emph{distributivity}) the set $\Sscr$ is nonempty and  $(\Sscr,\prec_F)$ forms a distributive lattice (cf. Theorems~1,8 in~\cite{AG});
\item[(b)]
(\emph{polarity}): $\prec_F$ is opposite to $\prec_W$ on $\Sscr$, i.e. for $x,y\in\Sscr$, if  $x_f\preceq_f y_f$ for all $f\in F$, then $y_w\preceq_w x_w$ for all $w\in W$, and vice versa (cf. Corollary~2 in~\cite{AG}).
 \end{itemize}
   \end{numitem1}

We denote the minimal and maximal element in the lattice $(\Sscr,\prec_F)$ by
$\xmin$ and $\xmax$, respectively; then the former is the best and the latter
is the worst for the part $W$, in view of the polarity~\refeq{AG}(b).
  \medskip
  
Next we review several ingredients from the work~\cite{karz} that will be essential for our further description.

%-----------------Subsec.3.1

\subsection{Rotations.} \label{ssec:rotat}
~By an (abstract) \emph{rotation} we mean a cycle $R=(v_0,e_1,v_1,\ldots, e_k,v_k=v_0)$ in $G$ in which the edges $e_1,\ldots,e_k$ are different, i.e. $R$ is edge-simple (but not necessarily simple, as it may have repeated vertices). We write $V_R$ and $E_R$ for the sets of (different) vertices and edges in $R$, respectively. An edge $e_i$ is called \emph{positive} (\emph{negative}) if the vertex $v_{i-1}$ belongs to $W$ (resp. to $F$); in the former (latter) case, we also say that the edge $e_i$ \emph{is directed} from $W$ to $F$ (resp. from $F$ to $W$). The set of positive (negative) edges of $R$ is denoted by $R^+$ (resp. $R^-$), and we associate to $R$ the incidence $0,\pm 1$ vector $\chi^R\in \Zset^E$ taking value 1 on the posititive edges, $-1$ on the negative edges, and 0 on the other edges of $G$. The \emph{reversed} rotation $(v_k,e_k,v_{k-1},\ldots, e_1,v_0=v_k)$ is denoted by $\bar R$; then $\bar R^+=R^-$, ~$\bar R^-=R^+$ and  $\chi^{\bar R}=-\chi^R$.

As is shown in~\cite{karz}, for each stable vector $x\in\Sscr$, there exists (and can be efficiently constructed) a set $\Rscr(x)$ of rotations possessing the following nice properties:
  \begin{numitem1} \label{eq:rot_prop} 
    \begin{itemize}
\item[(i)]
for each $R\in\Rscr(x)$, the vector $x':=x+\chi^R$ is stable and it immediately succeeds $x$ in the lattice $(\Sscr,\prec_F)$, in the sense that $x\prec_F x'$ and there is no $y\in\Sscr$ such that $x\prec_F y\prec_F x'$ (cf.~\cite[Prop.~3.4]{karz});
 \item[(ii)]
conversely, for each stable vector $x'$ immediately succeeding $x\in\Sscr$, there exists a rotation $R\in\Rscr(x)$ such that $x'=x+\chi^R$ (cf.~\cite[Prop.~3.5]{karz});
 \item[(iii)] the edge sets of rotations in $\Rscr(x)$ are pairwise disjoint (cf.~\cite[Sec.~3.1]{karz}).
 \end{itemize}
 \end{numitem1}

A rotation $R\in\Rscr(x)$ is called  \emph{applicable} to $x\in\Sscr$; we also say that the stable vector $x':=x+\chi^R$ is obtained from $x$ by applying the rotation $R$ (with weight 1) and that $R$ is \emph{increasing}. (This term respects the order $\prec_F$, in view of $x\prec_F x'$.) Application of the reversed rotation $\bar R$ to $x'$ returns $x$, namely, $x=x'+\chi^{\bar R}$, and we say that $\bar R$ is \emph{decreasing} (w.r.t. $\prec_F$). 

Clearly $\xmax$ (resp. $\xmin$) is the unique stable vector admitting no increasing (resp. decreasing) rotation.
 \medskip
 
\noindent\textbf{Remark 1.} 
An increasing rotation $R$ applicable to $x\in\Sscr$ is defined up to shifting cyclically. In fact, a priori we cannot exclude the existence of another rotation $R'$ (applicable to some $y\ne x$) having the same positive and negative parts: $R'^+=R^+$ and $R'^-=R^-$. A useful property shown in~\cite[Sec.~3.1]{karz} is that the cycle (increasing rotation) $R$ is determined by $x$ and one edge $e$ in it. More precisely, one shows that 
  \begin{numitem1} \label{eq:e-ep}
if $e$ is a positive edge of an increasing rotation $R$ connecting vertices $w\in W$ and $f\in F$, then $e$ is interesting for $f$ under $x$ and satisfies $C_f(x_f+\onebf^e_f)=x_f+\onebf^e_f-\onebf^{e'}_f$, where $e'$ is the next (negative) edge in the cycle $R$;
  \end{numitem1}
here for $e''\in E_f$, we write $\onebf^{e''}_f$ for its incidence vector in $\Rset^{E_f}$ (taking value 1 on $e''$, and 0 otherwise). When $e$ is negative, the (positive) edge $e'$ in $R$ next to $e$ is determined in a somewhat different way; see~\cite[Ex.~(3.3)]{karz}. (Also this $(e,e')$ provides the relation $C_w(x'_w+\onebf_w^{e})=x'_w+\onebf_w^{e}-\onebf_w^{e'}$, where $x':=x+\chi^R$.) Thus, on this way, we can restore the entire cycle $R$. In other words, for a pair $x,x'$ such that $x'$ immediately succeeds $x$ in $(\Sscr,\prec_F)$, to compute the (unique) rotation $R$  of which application to $x$ gives $x'$, it suffices to choose any edge $e$ with $x(e)\ne x'(e)$ and then proceed as above. 

Regarding a stable $x'$ and the order $\prec_W$ and acting in a similar fashion, we can restore (starting with a single edge) the corresponding decreasing rotation serving to transfer $x'$ into $x$; the obtained rotation is just $\bar R$ reversed to $R$ as above. 

Note also that, using~\refeq{e-ep}, one shows that 
  \begin{numitem1} \label{eq:CfR+}
for an increasing rotation $R$ applicable to a stable $x$ and passing a vertex $f\in F$, there holds $C_f(x_f+\chi_f^{R^+})=x_f+\chi_f^{R^+}-\chi_f^{R^-}$, where $\chi_f^{R^+}$ (resp. $\chi_f^{R^-}$) is the sum of vectors $\onebf_f^e$ over $e\in R^+\cap E_f$ (resp. $\chi_f^{R^-}\cap E_f$) 
\end{numitem1}
(cf.~\cite[Lem.~3.6]{karz}). Properties~\refeq{e-ep} and~\refeq{CfR+} will be used in Sect.~\SEC{nonbipartite}.

%-----------------Subsec.3.2

\subsection{Weights of rotations.} \label{ssec:weight_rot}
~For an increasing rotation $R$ applicable to a stable $x\in\Sscr$, we may try to apply $R$ with a larger weight. More precisely, we say that a weight $\lambda\in\Zset_{>0}$ is \emph{feasible} for $(x,R)$ if $R$ can be applied with weight 1, step by step, $\lambda$ times, which means that the sequence $x=x_0,x_1,\ldots,x_\lambda$ defined by $x_i:=x_{i-1}+\chi^R$, $i=1,\ldots,\lambda$, consists of stable vectors. (One shows that in this case $x_{i-1}\prec_F x_i$ is valid for each $i$.) The maximal feasible weight for $(x,R)$ is denoted by $\tau_R(x)$.

Consider the set $\Rscr(x)$ of (increasing) rotations applicable to $x\in\Sscr$. An important fact is that the rotations in $\Rscr(x)$ commute. Moreover, the following property is valid (see~\cite[Cor.~4.1]{karz}).
  \begin{numitem1} \label{eq:commute}
Let $\Rscr'\subseteq\Rscr(x)$ and let $\lambda:\Rscr'\to \Zset_+$ be such
that $\lambda(R)\le \tau_R(x)$ for each $R\in\Rscr'$. Then the vector
$x':=x+\sum(\lambda(R) \chi^R \colon R\in\Rscr')$ is stable, each $R\in \Rscr'$ with $\lambda(R)<\tau_R(x)$ is an increasing rotation applicable to $x'$ having the maximal feasible weight $\tau_R(x')=\tau_R(x)-\lambda(R)$, and each $R'\in \Rscr(x)-\Rscr'$ is applicable to $x'$ keeping the maximal feasible weight: $\tau_{R'}(x')=\tau_{R'}(x)$. In particular, rotations in $\Rscr(x)$ can be applied in any order.  
  \end{numitem1}

%-----------------Subsec.3.3

\subsection{Routes.} \label{ssec:routes}
~Let $\Tscr$ be a sequence $x_0,x_1,\ldots,x_N$ of stable vectors such that each $x_i$ is obtained from $x_{i-1}$ by applying an increasing rotation $R_i$ with a feasible weight $\lambda_i\in\Zset_{>0}$, i.e. $R_i\in\Rscr(x_{i-1})$, $\lambda_i\le \tau_{R_i}(x_{i-1})$ and $x_i=x_{i-1}+\lambda_i\chi^{R_i}$. Then $x_0\prec_F x_1\prec_F\cdots \prec_F x_N$. We call $\Tscr$ a \emph{route} from $x_0$ to $x_N$. From~\refeq{rot_prop} it follows that there exists a route from $\xmin$ to $\xmax$ (when all $\lambda_i$'s are ones, this is analogous to a maximal chain in a finite lattice). We liberally say that a rotation $R_i$ with weight $\lambda_i$ \emph{is used} in $\Tscr$.

Note that one and the same rotation may be used in a route many times. Using terminology from~\cite{karz}, a route $\Tscr$ as above  is called: \emph{non-excessive} if $i<j$ and $R_i=R_j$ imply $\lambda_i=\tau_{R_i}(x_{i-1})$;  and \emph{principal} if $\lambda_i=\tau_{R_i}(x_{i-1})$ for all $i=1,\ldots,N$. A principal route from $\xmin$ to $\xmax$ is called \emph{full}. Note that using~\refeq{commute}, one shows that for any $x,y\in \Sscr$ with $x\prec_F y$, there exists a non-excessive route from $x$ to $y$.% (cf.~\cite[Ex.~(4.3)]{karz}).

For a non-excessive route $\Tscr$, let $\Pi(\Tscr)$ denote the \emph{family} of pairs $(R_i,\lambda_i)$ (\emph{weighted rotations}) used in $\Tscr$ (where each pair $(R,\lambda)$ occurs in $\Pi$ as many times as it is used in $\Tscr$). The following property is of importance (see~\cite[Prop.~4.2]{karz}):
  \begin{numitem1} \label{eq:invar_rot}
Let $x,y\in \Sscr$ and $x\prec_F y$. Then for all non-excessive routes $\Tscr$ going from $x$ to $y$, the family $\Pi(\Tscr)$ is the same. A similar property is valid relative to principal routes as well (when $y$ is reachable from $x$ by a principal route).
 \end{numitem1}
\noindent(Initially an invariance property of this sort (in case $(x,y)=(\xmin,\xmax)$) was revealed by Irving and Leather~\cite{IL} for usual rotations in the classical stable marriage problem and subsequently was  demonstrated for more general models of stability. This analogous to the fundamental fact due to Birkhoff~\cite{birk} that the set of prime factors associated with a maximal chain in a finite distributive lattice does not depend on the chain.)

Let us call a stable vector $x\in \Sscr$ \emph{principal} if there is a principal route from $\xmin$ to $x$. We denote the set of principle vectors by $\Sscrpr$, and the restriction of $\prec_F$ to $\Sscrpr$ by $\precpr$. (One easily shows that $(\Sscrpr,\precpr)$ forms a distributive sublattice of $(\Sscr,\prec_F)$.) 

%-----------------Subsec.3.4

\subsection{Poset of rotations and closed functions.} \label{ssec:poset_rot}
~By~\refeq{invar_rot}, the family $\Pi(\Tscr)$ of pairs (weighted rotations) $(R,\tau)$ that are used (respecting possible replications) in a full route $\Tscr$ does not depend on the route; we abbreviate it as $\Pi$. We write $\Rscr$ for the set of \emph{different} rotations used in a full route; equivalently, $\Rscr$ is the set of  different cycles in $G$ forming increasing rotations applicable to stable vectors. For $R\in\Rscr$, the subfamily of pairs in $\Pi$ involving this $R$ is denoted by $\Pi_R$. We also denote by $\hat\Rscr$ the family of all rotations, with possible replications, occurring in $\Pi$ (then for $R\in\Rscr$, there are as many occurrences of $R$ in $\hat\Rscr$ as the cardinality of $\Pi_R$).

 Next we arrange a poset on $\Pi$ that gives rise to a representation for the principal lattice $(\Sscrpr,\precpr)$ and, further, for the whole lattice $(\Sscr,\prec)$ (extending a classical result by Irving et al.~\cite{ILG} on a poset representation for stable marriages). More precisely (see~\cite[Th.~5.9, Cor.~5.10, Ex.~(5.6)]{karz}),
  \begin{numitem1} \label{eq:princ_biject}
there exist a partial order $\lessdot$ on $\Pi$ and a map $\phipr$ from the principal vectors to subsets of $\Pi$ such that:
 \begin{itemize}
 \item[(a)] for each $x\in\Sscrpr$, $\phipr(x)$ is the family of weighted rotations $(R,\tau)$ used in a principal route from $\xmin$ to $x$ (i.e.  $x=\xmin+\sum(\tau\chi^R \colon (R,\tau)\in \phipr(x))$);
 \item[(b)] $\phipr$ establishes an isomorphism between the principal lattice $(\Sscrpr,\precpr)$ and the lattice of closed subfamilies in $(\Pi,\lessdot)$, where $\Cscr\subseteq \Pi$ is called \emph{closed} if $(R,\tau),(R',\tau')\in \Pi$ and $(R',\tau')\in \Cscr$ imply $(R,\tau)\in\Cscr$ (and the closed subfamilies are partially ordered by inclusion);
\item[(c)] for each rotation $R\in\Rscr$, the restriction of $\lessdot$ to $\Pi_R$ is a linear order.
 \end{itemize}
 \end{numitem1}
 
\noindent In particular, under the bijection $\phipr$ between the principal stable vectors and the closed families in $(\Pi,\lessdot)$, $\xmin$ corresponds to the empty set, and $\xmax$ to the whole $\Pi$. 

Equivalently, the partial order $\lessdot$ can be defined as follows (in a spirit of~\cite{ILG}):
 \begin{numitem1} \label{eq:lessdot}
pairs $(R,\tau),(R',\tau')\in\Pi$ satisfy $(R,\tau)\lessdot (R',\tau')$ if and only if for \emph{any} full route $\Tscr$, the pair $(R,\tau)$ is used in $\Tscr$ \emph{earlier} than $(R',\tau')$ (see~\cite[Sects.~5.1,5.2]{karz}).
  \end{numitem1}

Next we extend $\phipr$ to a map from all stable vectors. Let us call a function $\lambda:\Pi\to\Zset_+$ \emph{closed} if it does not exceed $\tau$ (i.e. $\lambda(R,\tau)\le \tau$ for all $(R,\tau)\in\Pi$) and the relations $(R,\tau)\lessdot (R',\tau')$ and $\lambda(R',\tau')>0$ imply $\lambda(R,\tau)=\tau$. 

One can see that for any closed function $\lambda$, the family $\{(R,\tau)\in\Pi\colon \lambda(R,\tau)>0\}$ is closed, and ``conversely'': taking a closed family $\Cscr\subseteq\Pi$, one can form a closed function $\lambda$ by defining $\lambda(R,\tau)$ to be an arbitrary integer between 0 and $\tau$ for each \emph{maximal} pair $(R,\tau)$ in $\Cscr$, the value $\tau$ for the other pairs $(R,\tau)$ in $\Cscr$, and 0 for the rest. As to stable vectors, for any $x\in\Sscr$, taking a non-excessive route $\Tscr$ from $\xmin$ to $x$ and relying on~\refeq{commute}, one can transform $\Tscr$ into a principal route in a natural way, by assigning the maximal feasible weight of a current rotation at each step, thus associating to $x$ the corresponding principal vector $x'$. Based on these observations, the following facts can be concluded (see~\cite[Th.~5.11]{karz}):
  \begin{numitem1} \label{eq:gen_biject}
there exists a map $\phi$ from the stable vectors to functions on $\Pi$ such that:
  \begin{itemize}
\item[(a)] for each $x\in\Sscr$, $\phi(x)$ is generated by a non-excessive route $\Tscr$ from $\xmin$ to $x$; namely, for each pair $(R,\tau)\in\Pi$, $\phi(x)$ takes value $\lambda$ if $R$ (as the corresponding copy in $\hat\Rscr$) is used with weight $\lambda$ in $\Tscr$, and 0 otherwise;
 \item[(b)] $\phi$ establishes an isomorphism between the lattice $(\Sscr,\prec)$ and the lattice of closed functions for $(\Pi,\lessdot)$ (under the natural comparison $\le$);
 \item[(c)] the restriction of $\phi$ to $\Sscrpr$ coincides with $\phipr$. 
 \end{itemize}
  \end{numitem1} 
  
%-----------------Subsec.3.5

\subsection{Complexities.} \label{ssec:complex}
~As is said in the Introduction, we assume that each choice function $C_v$ is given via an \emph{oracle} that, being asked of a vector $z\in\Bscr_v$, outputs its ``acceptable part'' $C_v(z)$. Each algorithm (procedure) arising in our constructions is either \emph{pseudo}, or \emph{weakly}, or \emph{strongly} polynomial, which means (in our case) that its running \emph{time}, estimating the number of standard operations plus oracle calls, is bounded by a polynomial in $|V|$ and $\bmax$, or in $|V|$ and $\log \bmax$, or in $|V|$, respectively, where $\bmax$ is the maximum capacity of an edge. The term \emph{efficient} is applied to weakly and strongly polynomial algorithms. Like in~\cite{karz}, we usually do not care of precisely estimating time bounds for our algorithms and restrict ourselves by merely establishing their pseudo, weakly, or strongly polynomial-time complexity. Also we usually estimate the number of oracle calls rather than the complexity of their implementations.

Basic tasks used in our constructions concern finding rotations and their maximal weights. As is shown in~\cite[Sec.~6]{karz},
  \begin{numitem1} \label{eq:2tasks}
given a stable vector $x\in\Sscr$, 
  \begin{itemize}
  \item[(a)] the set $\Rscr(x)$ of rotations applicable to $x$ can be constructed in strongly polynomial time; in particular, it takes $O(|E|^2)$ oracle calls;
  \item[(b)] for each $R\in\Rscr(x)$, the maximal feasible weight $\tau_R(x)$ can be found in pseudo polynomial time; in particular, it takes $O(\bmax|E|)$ oracle calls.
    \end{itemize}
    \end{numitem1}
    
An important parameter in our computations is the length $N$ of a full route $\Tscr=(x_0,x_1,\ldots,x_N)$, which is equal to the number $|\Pi|$ of elements in the poset $(\Pi,\lessdot)$. One shows that
  \begin{numitem1} \label{eq:N}
  \begin{itemize}
 \item[(a)] the poset $(\Pi,\lessdot)$ can be constructed in pseudo polynomial time; more precisely, the number of oracle calls can be estimated as $O(|E|^2(\bmax N+N^2))$, and the number of other (standard) operations as $O(\bmax N+N^2\log N)$ times a polynomial in $|V|$ (see~\cite[Th.~6.1]{karz});
\item[(b)] $N\le \bmax |E|/2$. 
   \end{itemize} 
   \end{numitem1}
Here~(b) is a consequence from the following useful property (see the proof of Lemma~6.2 in~\cite{karz}) :
   \begin{numitem1} \label{eq:one_peak}
for $e\in E$ and a full route $\Tscr=(x_0,x_1,\ldots,x_N)$, the values $x_0(e),x_1(e),\ldots,x_N(e)$ change in a ``one-peak'' manner; namely, there is $i$ such that the values from $x_0(e)$ to $x_i(e)$ are weakly increasing (admitting equalities), and the values from $x_i(e)$ to $x_N(e)$ are weakly decreasing.
  \end{numitem1}
   
As is mentioned in the Introduction, the bound $\Omega(\bmax)$ for $N$ is attained by an example in~\cite[Sec.~8]{karz}. In that example: the input graph $G$ has only six vertices; the maximum capacity $\bmax$ can be arbitrarily large; the set $\Rscr$ consists of only two rotations; these rotations are alternately used in the unique full route; and the length $N$ of this route is just $\bmax$. Accordingly, the Hasse diagram of the poset $(\Pi,\lessdot)$ forms the directed path with $\bmax$ vertices.
 \medskip
   
\noindent\textbf{Remark 2.}
In further description we will prefer to transfer the partial order $\lessdot$ from the pairs $(R,\tau)$ in $\Pi$ to the corresponding (``unweighted'') rotations $R$ in $\hat\Rscr$ (the family of occurrences of rotations used in a full route). Namely, for $R,R'\in\hat\Rscr$, we may write $R\lessdot R'$ if the corresponding pairs $(R,\tau),(R',\tau')$ in $\Pi$ satisfy $(R,\tau)\lessdot(R',\tau')$. On this way, we may use the alternative notation $(\hat\Rscr,\tau,\lessdot)$ for the poset $(\Pi,\lessdot)$, and  for a pair $(R,\tau)\in\Pi$, denote the (maximal) weight $\tau$ of $R$ by $\tau_R$. Also w.l.o.g. we may consider closed subsets $\Rscr'$ in $\hat\Rscr$ and closed functions $\lambda$ on $\hat\Rscr$ in place of $\Pi$.

%----------------------- Sec.4

\section{Features of the poset in the symmetric bipartite case} \label{sec:symmetr}

In this section we consider symmetric $G^\symm=(V^\symm=V^0\cup V^1,\, E^\symm), b^\symm,C^\symm$ as in the end of Sect.~\SEC{defin} and utilize terminology and results from Sect.~\SEC{back}. To match notation in Sect.~\SEC{back}, the parts $V^0$ and $V^1$ are also denoted as $W$ (``workers'') and $F$ (``firms''), respectively.

Let $\sigma$ denote the symmetry operator acting on objects in $G^\symm$ and functions on $E^\symm$. This $\sigma$ maps a vertex $v^i\in V^i$ to $v^{1-i}$, an edge $u^0v^1\in E^\symm$ to $v^0u^1=u^1v^0$, a cycle $R=(v_0,e_1,v_1,\ldots,e_k,v_k=v_0)$ to the cycle $\sigma(R)=(\sigma(v_0),\sigma(e_1),\sigma(v_1),\ldots,\sigma(e_k),\sigma(v_k))$, a vector $x\in\Zset^{E^\symm}$ to the vector $\sigma(x)$ taking values $x(\sigma(e))$ for all $e\in E^\symm$. Clearly $\sigma$ is an involution, i.e. $\sigma^2={\rm id}$. For convenience we also use brief notation $x^\ast$ for $\sigma(x)$, and $R^\ast$ for $\sigma(R)$. One can see that the sets of positive (directed from $W$ to $F$) and negative (directed from $F$ to $W$) edges in the cycles $R$ and $R^\ast$ are related as follows:
  \begin{equation} \label{eq:R+R-}
  (R^\ast)^+=(R^-)^\ast\qquad \mbox{and}\qquad (R^\ast)^-=(R^+)^\ast.
  \end{equation}
  
Also the symmetry of choice functions $C^\symm_{v^0}$ and $C^\symm_{v^1}$ implies that 
  \begin{numitem1} \label{eq:x-xast}
  for each stable $x\in\Sscr$, the vector $x^\ast$ is stable as well.
  \end{numitem1}
  
As before, we denote the minimal and maximal element in the lattice $(\Sscr,\prec_F)$  by $\xmin$ and $\xmax$, respectively. The symmetry of $\Sscr$ and the polarity of $\prec_F$ and $\prec_W$ imply that $\xmax$ is symmetric to $\xmin$. One can see the following.
  \begin{lemma} \label{lm:x_prev_y}
Let $x,y\in\Sscr$ be such that $x\prec_F y$ and $y=x+\lambda\chi^R$, where $R$ is a rotation applicable to $x$ and $\lambda$ is a feasible weight for $x,R$. Then $y^\ast\prec_F x^\ast$ and $x^\ast=y^\ast +\lambda\chi^{R^\ast}$ (where both $R,R^\ast$ are increasing w.r.t. $\prec_F$).
 \end{lemma}
 \begin{proof}
~Form a route $\Tscr=(x_0,x_1,\ldots,x_N)$ from $x_0=\xmin$ to $x_N=\xmax$ that contains the consecutive pair $x,y$, say, $x=x_i$ and $y=x_{i+1}$. Then $x_0\prec_F x_1\prec_F\ldots \prec_F x_N$. This is symmetric to $\xmin=x^\ast_N\prec_F x^\ast_{N-1}\prec_F\ldots \prec_F x^\ast_0=\xmax$ (taking into account that $x_{j-1}\prec_F x_j$ implies $x^\ast_{j-1}\prec_W x^\ast_j$, and $\prec_F$ is polar to $\prec_W$). It follows that $y^\ast\prec_F x^\ast$ and that $x^\ast$ is obtained from $y^\ast$ by applying some rotation $R'$ with weight $\lambda$. 

To examine $R'$, consider a positive edge $e$ in $R$. Then $y(e)=x(e)+\lambda$. By symmetry, $y^\ast(\sigma(e))=x^\ast(\sigma(e))+\lambda$, whence $x^\ast(\sigma(e))=y^\ast(\sigma(e))-\lambda$. Similarly, for a negative edge $e$ in $R$, we obtain $x^\ast(\sigma(e))=y^\ast(\sigma(e))+\lambda$. So $R'^-=(R^+)^\ast$ and $R'^+=(R^-)^\ast$; cf.~\refeq{R+R-}. Moreover, by symmetry, the order of edges in the cycle $R'$ must be the same, up to reversing, as in the cycle $\bar R$ reverse of $R$, and we can conclude that $R'=R^\ast$ (cf.~reasonings in Remark~1 from Sect.~\SSEC{rotat}).
 \end{proof}
 
Using this lemma, we can obtain the following:
  \begin{numitem1} \label{eq:Tscr-ast}
if $\Tscr=(\xmin=x_0,x_1,\ldots,x_N=\xmax)$ is a full route, and each $x_i$ is obtained from $x_{i-1}$ by applying a rotation $R_i$ with weight $\tau_i$ ($=\tau_{R_i}(x_{i-1})$), then $(x_N^\ast,x^\ast_{N-1},\ldots,x^\ast_0)$ is a full route as well, denoted as $\Tscr^\ast$; furthermore, each $x^\ast_{i-1}$ is obtained from $x^\ast_i$ by applying the rotation $R_i^\ast$ with the same weight $\tau_i$.
  \end{numitem1}
  
\noindent As a consequence, the maximal feasible weights of related rotations $R$ and $R^\ast$ in these routes are equal; namely, if they belong to pairs $(R,\tau)$ and $(R^\ast,\tau')$ in $\Pi$, then $\tau=\tau'$.

Next we take advantages from the map $\phi$ establishing a bijection between the set $\Sscr$ of stable vectors and the set of closed functions $\lambda$ in the poset $(\Pi,\lessdot)$ (see~\refeq{gen_biject}(b))  in our symmetric case.
Like in Remark~2 in Sect.~\SSEC{complex}, we may equivalently assume that $\lambda$ is given on the family $\hat\Rscr$. Also we write $R\lessdot R'$ if the corresponding pairs $(R,\tau),(R',\tau')$ in $\Pi$ satisfy $(R,\tau)\lessdot(R',\tau')$. Recall (cf.~\refeq{gen_biject}(a)) that for a stable $x\in\Sscr$, the closed function $\phi(x)$ is defined by considering a non-excessive route $\Tscr$ from $\xmin$ to $x$, and putting the value of $\phi(x)$ on $R\in\hat\Rscr$ to be $\lambda$ if $R$ is used with weight $\lambda$ in $\Tscr$, and zero if $R$ is not used in $\Tscr$ (by~\refeq{invar_rot}, $\phi(x)$ does not depend on the choice of $\Tscr$).

We will need the following nice property:
  \begin{numitem1} \label{eq:antisym}
the partial order $\lessdot$ is ``antisymmetric'', in the sense that if rotations $R,R'\in\hat\Rscr$ satisfy $R\lessdot R'$, then $R^\ast\gtrdot R'^\ast$,
  \end{numitem1}
Indeed, relation $R\lessdot R'$ means that for any full route $\Tscr=(x_0,\ldots,x_N)$, $R$ is used in $\Tscr$ earlier than $R'$ (see~\refeq{lessdot}). Then in the symmetric full route $\Tscr^\ast=(x_N^\ast,\ldots,x_0^\ast)$, the rotation $R^\ast$ is used later than $R'^\ast$ (cf. reasonings in the proof of Lemma~\ref{lm:x_prev_y}).

 The next assertion is of importance.
 \begin{lemma} \label{lm:x-xstar}
For $x\in\Sscr$, consider the closed functions $\lambda:=\phi(x)$ and $\eps:=\phi(x^\ast)$. Then each rotation $R\in\hat\Rscr$ satisfies the equality $\lambda(R)+\eps(R^\ast)=\tau_R$, where $\tau_R$ is the weight of $R$ as in its pair in $\Pi$ (and $\tau_{R^\ast}=\tau_R$).
 \end{lemma}
  \begin{proof}
~Take a non-excessive route $\Tscr$ from $\xmin$ to $x$, and considering the reverse order $\prec_W$, form a non-excessive route $\Tscr'$ from $\xmax$ to $x$. Then the rotations used in $\Tscr$ are increasing, while the ones used in $\Tscr'$ are decreasing w.r.t. $\prec_F$. The concatenation of $\Tscr$ and the reverse of $\Tscr'$ gives a route $\Tscr''$ from $\xmin$ to $\xmax$.  

Let $\eps'$ denote the function of weights of (decreasing) rotations used in $\Tscr'$, extended by zero on the rotations not used in $\Tscr'$. We know that for any rotation $R\in\hat \Rscr$, the sum of weights in the occurrences  of $R$ used in a route from $\xmin$ to $\xmax$, in particular, in $\Tscr''$, is equal to $\tau_R$. It follows that $\lambda(R)+\eps'(\bar R)=\tau_R$ for all $R\in\hat\Rscr$, where $\bar R$ is the reverse of $R$ (see Sect.~\SSEC{rotat}). (In particular, if $R$ is used in $\Tscr$ and $\bar R$ is not used in $\Tscr'$, then $\lambda(R)=\tau_R$ and $\eps'(\bar R)=0$; and if $R$ is not used in $\Tscr$ and $\bar R$ is used in $\Tscr'$, then $\lambda(R)=0$ and $\eps'(\bar R)=\tau_R$.)

Under reversing $\Tscr'$ and replacing its elements (stable vectors) by the symmetric ones, we obtain a route $\Uscr$ from $\xmin$ to $x^\ast$. One can see that each decreasing rotation $\bar R=(v_0,e_1,v_1,\ldots,e_k,v_k)$ used in $\Tscr'$ turns into the increasing rotation $\rho$ of the form $(\sigma(v_k),\sigma(e_k),\sigma(v_{k-1}),\ldots, \sigma(e_1),\sigma(v_0))$ used in $\Uscr$ (yielding $\rho^+=(\bar R^\ast)^+$ and $\rho^-=(\bar R^\ast)^-$). This is nothing else that the rotation $R^\ast$ symmetric to $R$. Also the weight of $\rho=R^\ast$ used in $\Uscr$ remains the same as that of $\bar R$, namely, $\eps'(\bar R)$. Therefore, for the closed function $\eps=\phi(x^\ast)$, we have $\eps(R^\ast)=\eps'(\bar R)$, and now $\lambda(R)+\eps'(\bar R)=\tau_R$ gives the desired equality $\lambda(R)+\eps(R^\ast)=\tau_R$.  
  \end{proof}

In our further analysis, we distinguish \emph{self-symmetric} rotations $R\in\hat \Rscr$, i.e. those satisfying $R^\ast=R$; we also call them \emph{singular}. 

(This is equivalent to the fact that the cycle $R$ contains simultaneously two edges $e$ and $\sigma(e)$ such that the former is  positive (directed from $W$ to $F$) and the latter is negative (directed from $F$ to $W$). To see this, let $R=(v_0,e_1,v_1,\ldots,e_k,v_k=v_0)$ and assume that $e=e_1$ and $\sigma(e)=e_i$. Then $v_0\in W$,  $v_{i-1}=\sigma(v_0)\in F$, and the part $P=(v_0,e_1,v_1,\ldots, v_{i-1})$ of $R$ is symmetric to $P'=(v_{i-1},e_i,v_i,\ldots,v_k=v_0)$. One shows that $R$ is the concatenation of $P$ and $P'$ (it suffices to check, using reasonings from Remark~1, that the fact that $e_{i-1}, e_i$ are consecutive edges in a rotation for $x\in\Sscr$ implies a similar fact for $\sigma(e_{i-1})$ and $e_1$). Then $R$ is self-symmetric. Note also that the lengths of $P$ and $P'$ are equal and odd. Therefore, $R$ has length $k\equiv 2\,(\!\!\!\mod 4)$.)
 
  The simplest example is illustrated in the picture; here the rotation has length 6 and passes the vertices $a^0,b^1,c^0,a^1,b^0,c^1$ (the example arises from the triangle $abc$).

\vspace{-0.2cm}
\begin{center}
\includegraphics[scale=0.8]{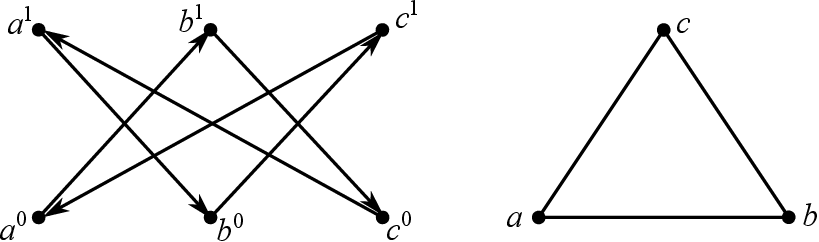}
\end{center}
\vspace{-0.2cm}

Note that if $R,R'$ are two different singular rotations, then $R\lessdot R'$ is impossible, for otherwise we would have $R=R^\ast\gtrdot R'^\ast=R'$ (using~\refeq{antisym}). Therefore,
  \begin{numitem1} \label{eq:sing_rot}
all singular rotations are incomparable by $\lessdot$.
 \end{numitem1}

Lemma~\ref{lm:x-xstar} enables us to obtain the following characterization of closed functions corresponding to symmetric stable vectors (this is a direct extension of Lemma~14 from Dean and Munshi~\cite{DM} on the non-bipartite allocation problem).
  \begin{corollary} \label{cor:symm_closed}
Under the map $\phi$, the set $\Sscr^\symm$ of symmetric stable vectors $x$ (i.e. satisfying $x=x^\ast$) is bijective to the set of closed functions $\lambda:\hat\Rscr\to \Zset_+$ such that 
 \begin{equation} \label{eq:lambdaRRast}
 \lambda(R)+\lambda(R^\ast)=\tau_R\qquad \mbox{for each} \quad R\in\hat\Rscr.
  \end{equation}
As a consequence, $\lambda(R)=\tau_R/2$ for each singular $R$, which implies that $\Sscr^\symm$ can be nonempty only if the weights $\tau_R$ of all singular $R$ are even.
 \end{corollary}

Indeed, if $x=x^\ast$, then $\lambda:=\phi(x)=\phi(x^\ast)$ satisfies~\refeq{lambdaRRast}, by Lemma~\ref{lm:x-xstar} with $\eps=\lambda$. Conversely, for a closed function $\lambda$ and the stable vector $x:=\phi^{-1}(\lambda)$, if $\lambda$ satisfies~\refeq{lambdaRRast}, then the closed function $\eps:=\phi(x^\ast)$ must coincide with $\lambda$, whence $x=x^\ast$.
 \medskip
 
\noindent\textbf{Definition.} 
We call a closed function $\lambda:\hat\Rscr\to\Zset_+$ satisfying~\refeq{lambdaRRast} \emph{balanced}. Denoting the set of singular rotations $R$ with $\tau_R$ odd (if exist) by $\frakO$, we say that a closed function $\lambda:\hat\Rscr\to\Zset_+$ is \emph{quasi-balanced} (briefly, a \emph{QB-function}) if 
%$\lambda(R)\in\{\lfloor \tau_R/2\rfloor,\lceil\tau_R/2\rceil\}$ for all $R\in\frakO$, and 
$\lambda(R)$ satisfies~\refeq{lambdaRRast} for all $R\in\hat\Rscr-\frakO$.
 \medskip

The existence of a balanced function implies that $\frakO=\emptyset$. Below (in Lemma~\ref{lm:quasi-bal}) we shall see that the evenness of weights $\tau_R$ for all singular $R$ ensures that a balanced closed function does exist. This is based on the following algorithm that finds, as a by-product, a QB-function for arbitrary $G^\symm,b^\symm,C^\symm$. (Here we act like in a method for the non-bipartite allocation problem outlined in the proof of Lemma~15 in~\cite{DM}.)
 \medskip
 
\noindent\textbf{Algorithm QB.}
It constructs a certain route $\Tscr$ starting with $x:=\xmin$. At each iteration, the current vector $x$ is handled as follows. We scan the set $\Rscr(x)$ of rotations applicable to $x$, and seek for a rotation $R\in\Rscr(x)$ such that its symmetric $R^\ast$ is not used in the current route $\Tscr$ from $\xmin$ to $x$. If such an $R$ is found, then we assign the weight $\lambda(R)$ to be $\lfloor\tau_R/2\rfloor$ or $\lceil\tau_R/2\rceil$ if $R$ is singular, and to be $\tau_R$ otherwise. Then we update $x:=x+\lambda(R)\chi^R$, completing the iteration and regarding $R$ as being used. The algorithm terminates when the current $x$ is such that for each $R\in\Rscr(x)$, the rotation $R^\ast$ is already used
in $\Tscr$.
  \medskip
  
Note that at each iteration of the algorithm, the current function $\lambda$ is closed. Indeed, if a rotation $R$ is used in the current route $\Tscr$ and if $\lambda(R)<\tau_R$ (which is possible only if $R$ is singular), then $R$ must belong to the current set $\Rscr(x)$ (cf.~\refeq{commute}). All rotations in $\Rscr(x)$ are incomparable by $\lessdot$ (as they can be applied in any order, by~\refeq{commute}). It follows that any rotation $R'\in\hat\Rscr$ with $R'\lessdot R$ is already used in $\Tscr$ and satisfies $\lambda(R')=\tau_{R'}$.

As a consequence, the current $x$ is stable and $\lambda=\phi(x)$.

\begin{lemma} \label{lm:quasi-bal}
Let $x$ be the stable vector upon termination of Algorithm QB. Then the function $\lambda:=\phi(x)$ is quasi-balanced. As a consequence, $x$ is symmetric if $\frakO=\emptyset$.
  \end{lemma}
  \begin{proof}
~Let $\Rscr'$ be the family of rotations used in the final route $\Tscr$. It suffices to show that for any rotation $R\in\hat\Rscr$, at least one of $R$ and $R^\ast$ is contained in $\Rscr'$.  

Suppose this is not so. Then the set $\Rscr(x)$ (which is nonempty since, obviously, $x\ne\xmax$) contains only rotations $\rho$ such that $\rho^\ast\in\Rscr'$. Consider a minimal (by $\lessdot$) element $R$ in $\hat\Rscr-(\Rscr'\cup\Rscr(x))$. Then $R$ immediately succeeds some $\rho\in\Rscr(x)$. (For otherwise either $R$ is minimal in $\hat\Rscr$, or all rotations immediately preceding $R$ are in $\Rscr'-\Rscr(x)$; in both cases, $R$ must be applicable to $x$, whence $R\in\Rscr(x)$.)

Now $\rho\,\lessdot\, R$ implies $R^\ast\lessdot\rho^\ast$, by~\refeq{antisym}. But this is impossible since $\Rscr'$ is closed, $\rho^\ast\in\Rscr'$ and $R^\ast\notin\Rscr'$. (The fact that $\Rscr'$ is closed, i.e. obeys relation $\pi\in\Rscr'\,\&\,\pi'\lessdot\pi\Rightarrow \pi'\in\Rscr'$, follows from the closedness of $\lambda$.)  
  \end{proof}

Next we estimate the complexity of Algorithm QB. We assume availability of the Hasse diagram $\Hscr=(\Pi,\Escr)$ of the poset $(\Pi,\lessdot)$ (which is constructed in pseudo polynomial time, as indicated in~\refeq{N}). (Recall that $\Hscr$ is induced by the pairs $((R,\tau),(R',\tau'))$ where $(R,\tau)$ immediately precedes $(R',\tau')$.) It is convenient to fix a linear order on $E$, and regarding each rotation $R$ as a cycle, denote by $e(R)$ the first edge in $R$ under this order. 

First of all we need to find the symmetry relation on $\hat\Rscr$. This is done by the following simple procedure. We rely on the anti-symmetry~\refeq{antisym} of the poset. At the first step, we select the sets $\Mscr_1$ and $\Mscr'_1$ of minimal and maximal elements (vertices) in $\Hscr$, respectively. Then for each rotation $R\in\Mscr_1$, its symmetric $R^\ast$ is contained in $\Mscr'_1$, and vice versa. Note that $\Mscr_1$ is exactly the set of rotations applicable to $\xmin$, i.e. $\Mscr_1=\Rscr(\xmin)$. Therefore, the rotations in $\Mscr_1$ are pairwise edge-disjoint (cf.~\refeq{rot_prop}(iii)), whence the total number $m_1$ of edges in these rotations is $O(|E|)$. And similarly for the rotations in $\Mscr'_1$. Then for each $R\in\Mscr_1$, the edge symmetric to $e(R)$ belongs to exactly one rotation in $\Mscr'_1$, and this rotation is just $\Rscr^\ast$. It follows that to establish the symmetry relations between $\Mscr_1$ and $\Mscr'_1$ takes (roughly) $O(|\Mscr_1|\,m_1)$ time.

At the second step, we select the sets $\Mscr_2$ and $\Mscr'_2$ of minimal and maximal elements in $\Rscr^1:=\hat\Rscr-(\Mscr_1\cup\Mscr'_1)$, respectively. Then $\Mscr'_2$ is symmetric to $\Mscr_2$. Since the set $\Mscr_1$ is, obviously, closed, it generates a principal stable vector $x$ (cf.~\refeq{princ_biject}(b)), and we can see that $\Rscr(x)=\Mscr_2$. Then the total number $m_2$ of edges occurring in $\Mscr_2$ is $O(|E|)$, and acting as at the previous step, we establish the symmetry relations between $\Mscr_2$ and $\Mscr'_2$, in $O(|\Mscr_2|\,m_2)$ time.

Then we treat the set $\Rscr^2:=\Rscr^1-(\Mscr_2\cup\Mscr'_2)$ in a similar way, and so on. Eventually, when the current family $\Rscr^k$ becomes empty, we obtain all symmetric pairs $\{R,R^\ast\}$ with $R\ne R^\ast$ and all singular rotations $R=R^\ast$. The total time of the procedure can be roughly estimated as $O(N\,|E|)$, where $N=|\hat\Rscr|$ (not exceeding $\bmax|E|$, by~\refeq{N}(b)).

One more useful observation is as follows. For $R\in\hat\Rscr$, define $\Lscr(R)$ (resp. $\Uscr(R)$) to be the family of rotations $R'$ with $R'\lessdot R$ (resp. $R'\gtrdot R$). Let $\frakS$ denote the set of singular rotations. By~\refeq{sing_rot}, all elements of $\frakS$ are incomparable. Then the closed set $\frakL:=\cup(\Lscr(R)\colon R\in\frakS)$ generates the principal vector $x$ such that the set $\Rscr(x)$ coincides with $\frakS$. This implies (by~\refeq{rot_prop}(iii)) that
  \begin{numitem1} \label{eq:disjoint_sing}
the edge sets of singular rotations are pairwise disjoint, and therefore, the total number of edges in singular rotations is $O(|E|)$.
  \end{numitem1}

Now we describe a rather straightforward implementation of Algorithm QB (in a slightly modified form). Here we think of the rotations simply as the (enumerated) vertices of the graph $\Hscr$ rather than cycles.  Note that the values of $\lambda$ are already determined on $\frakS$ (where $\lambda(R)\in\{\lfloor \tau_R/2\rfloor,\lceil\tau_R/2\rceil\}$ for all $R\in\frakS$), on $\frakL$ (where $\lambda(R)=\tau_R$ for each $R\in\frakL$), and on $\frakU:=\cup(\Uscr(R)\colon R\in\frakS)$ (where $\lambda(R)=0$ for all $R\in\frakU$). 

So we can remove the vertex set $\frakS\cup\frakL\cup\frakU$ from $\Hscr$, and proceed with the resulting graph $\Hscr'=(\Rscr',\Escr')$ (where for convenience the vertex set is regarded as consisting of rotations $R$ rather than pairs $(R,\tau)$). We further reduce $\Hscr'$, step by step, maintaining the following structures: (a) for each $R\in\Rscr'$, the lists $P(R)$ and $S(R)$ of rotations immediately preceding and succeeding $R$ in the current $\Hscr'$, respectively, and the number (``counter'') $\zeta(R):=|P(R)|$; (b) the family $\frakC$ of rotations that have been removed from the initial $\Hscr'$, and the family $\Fscr\subseteq\Rscr'$ of rotations symmetric to those in $\frakC$; and (c) the set $\Mscr$ of rotations $R\in\Rscr'-\Fscr$ that are minimal in the current $\Hscr'$, or, equivalently, such that $R\notin\Fscr$ and $\zeta(R)=0$.

At an iteration, if the current $\Mscr$ is nonempty, we choose an arbitrary $R\in\Mscr$ and do the following: (i) put $\lambda(R):=\tau_R$ and $\lambda(R^\ast):=0$; (ii) for each $R'\in S(R)$, delete $R$ from the list $P(R')$ and decrease the counter $\zeta(R')$ by one; (iii) remove $R$ from $\Hscr'$ (and $\Mscr$), accordingly adding it to $\frakC$, and simultaneously make $R^\ast$ frozen (inserting it in $\Fscr$); and (iv) add to $\Mscr$ each element $R'$ from $S(R)$ whose counter $\zeta(R')$ becomes zero.

The process terminates when the current $\Mscr$ becomes empty. A straightforward examination shows that the final family $\frakC$ (consisting of the rotations $R$ with $\lambda(R)=\tau_R$) is closed in the initial $\Hscr'$, their symmetric rotations $R^\ast$ have zero weights: $\lambda(R^\ast)=0$, and all minimal rotations in the final $\Hscr'$ are frozen and their symmetric ones belong to $\frakC$. Arguing as in the proof of Lemma~\ref{lm:quasi-bal}, one can see that all elements in the final $\Hscr'$ are frozen, implying that $\lambda$ is balanced within the initial $\Hscr'$. As a consequence, we obtain that the function $\lambda$ on the whole $\hat\Rscr$ is quasi-balanced in $\Hscr$, as required. The above implementation is of linear complexity $O(|\hat\Rscr|+|\Escr|)$, and using~\refeq{disjoint_sing}, we can summarize the above observations as follows.
   \begin{theorem} \label{tm:algQB}
When the rotational poset $(\Pi,\lessdot)$ for $G^\symm,b^\symm,C^\symm$ is available, Algorithm QB (being implemented as above) finds a quasi-balanced closed function $\lambda$ on $\Pi$ in time $O(|\Pi| |E|)$, or $O(\bmax|E|^2)$. Also:
  \begin{itemize}
  \item[\rm(i)] for any quasi-balanced closed function $\lambda$, the following values are fixed: $\lambda(R)=\tau_R$ for all $R\in\frakL$; $\lambda(R)=0$ for all $R\in\frakU$; and $\lambda(R)=\tau_R/2$ for all $R\in\frakS-\frakO$  (where $\frakS,\frakL,\frakU,\frakO$ are defined above);
\item[\rm(ii)] the singular rotations are incomparable by $\lessdot$ and pairwise edge-disjoint;
  \item[\rm(iii)] if the weights $\tau_R$ of all $R\in\frakS$ are even (and only in this case), then each quasi-balanced closed function $\lambda$ is balanced, the vector $x:=\phi^{-1}(\lambda)$ is stable and symmetric for $(G^\symm,b^\symm,C^\symm)$, and $\beta^{-1}(x)$ is stable for the original $(G,b,C)$ (cf.~\refeq{biject}).
  \end{itemize}
  \end{theorem}
(Here (i) relies on the facts that $\lambda$ is closed and obeys~\refeq{lambdaRRast}, and that for any $R\in \frakL$, its symmetric $R^\ast$ belongs to $\frakU$ and satisfies $R\lessdot R^\ast$.) 

We complete this section with some remarks and open questions.
 \medskip
  
\noindent\textbf{Remark 3.}
Recall that by the definition of symmetric rotations, if an edge $e$ belongs to a rotation $R$, then the symmetric edge $\sigma(e)$ belongs to $R^\ast$ and these edges have different signs in these rotations: say, when $e$ is positive in $R$ (i.e. directed from $W$ to $F$), $\sigma(e)$ is negative in $R^\ast$ (directed from $F$ to $W$). We know that if both $e,\,\sigma(e)$ belong to $R$ and have \emph{different signs} in it, then $R$ is singular: $R=R^\ast$. Let us call a rotation $R$ \emph{regular} if $E_R\cap E_{R^\ast}=\emptyset$ (where $E_{R'}$ is the edge set of a rotation $R'$). One can ask: does there exist a rotation $R\in\hat\Rscr$ that is neither regular nor singular? (It is open to us at present.) This is equivalent to the property that $R$ contains a pair of symmetric edges $e,\sigma(e)$ that have \emph{the same sign} in $R$, say, both are positive in it. Then both $e,\sigma(e)$ belong to $R^\ast$ and are negative in it. In this case, application of $R$ to a stable vector $x$ increases $x(e)$, while application of $R^\ast$ decreases the value at $e$. This implies, by~\refeq{one_peak}, that in any full route, $R$ is used earlier than $R^\ast$, whence $R\lessdot R^\ast$. 
 \medskip
 
\noindent\textbf{Remark~4.} 
Let $\lambda$ be a QB-function on $\hat\Rscr$ such that: %$\lambda(R)=\tau_R$ for all $R\in\frakL$; $\lambda(R)=0$ for all $R\in\frakU$; 
$\lambda(R)\in\{\tau_R,0\}$ for all $R\in \hat\Rscr-\frakS$; and $\lambda(R)\in\{\lfloor \tau_R/2\rfloor,\lceil\tau_R/2\rceil\}$ for all $R\in\frakS$ (recall that the values of $\lambda$ on $\frakL\cup\frakU$ are fixed, cf. Theorem~\ref{tm:algQB}(i)). We call such a $\lambda$ a \emph{principal} QB-function. In particular, so are the $\lambda$'s constructed by Algorithm~QB. One can address the following issue: can one modify this algorithm so as to (efficiently) enumerate all principal QB-functions (up to their values on $\frakO$)? This is equivalent to enumerating all principal (fully balanced) functions within the subgraph $\Hscr'=(\Rscr',\Escr')$ of $\Hscr$ induced by the family $\hat\Rscr-\{\frakS\cup\frakL\cup\frakU)$. It is tempting to accomplish this task with $O(|E|)$ amortized time per one function by acting in a DFS manner or so.

%-------------------------- Sec 5

\section{Transferring to the non-bipartite case} \label{sec:nonbipartite}
 
In this section we use observations and results from the previous section to show the existence of a certain structure (generalizing stable partitions in~\cite{tan} and stable half-partnerships in~\cite{flein}) when the ground graph $G=(V,E)$ is non-bipartite and a stable partnership need not exist. 
 
Before doing this, we make an additional observation on (closed) quasi-balanced functions $\lambda$ in the above poset $(\Pi,\lessdot)$ for $G^\symm,b^\symm,C^\symm$. The facts that $\lambda$ takes zero values on the family $\frakU$ (cf.~Theorem~\ref{tm:algQB}(i)) and the singular rotations are incomparable imply that the values $\lambda(R)$ on singular rotations $R\in\frakO$ (with $\tau_R$ odd) can be changed arbitrarily within the integer intervals $[0,\tau_R]$, yielding again a  QB-function. Let us say that a QB-function $\lambda$ is \emph{lower} (\emph{upper}) at $R\in\frakO$ if $\lambda=\lfloor\tau_R/2\rfloor$ (resp. $\lambda=\lceil\tau_R/2\rceil$).  One can see (using Lemma~\ref{lm:x-xstar}) that 
 \begin{numitem1} \label{eq:lower}
if a  quasi-balanced $\lambda$ is lower at $R\in\frakO$, then the stable vectors $\tilde x:=\phi^{-1}(\lambda)$, ~$\tilde x^\ast$ and $\tilde x':=\tilde x+\chi^R$ coincide on the edges in $E^\symm- E^\symm(\frakO)$, and there hold: $\tilde x^\ast(e)=\tilde x'(e)=\tilde x(\sigma(e))=\tilde x(e)+1$ for all $e\in R^+$, and $\tilde x^\ast(e)=\tilde x'(e)=\tilde x(\sigma(e))=\tilde x(e)-1$ for all $e\in R^-$.
 \end{numitem1}
Hereinafter, to distinguish between vectors on $E$ and $E^\symm$, we denote the latter ones with tildes. Also we denote by $E^\symm(\Rscr')$ the set of edges covered by rotations in $\Rscr'\subseteq\hat\Rscr$.

Next we introduce the notion of half-partnership for our case of $G,b,C$. This structure involves an integer vector $x\in\Zset^E_+$ with $x\le b$ and a set $\Kscr$ of pairwise edge-disjoint cycles such that each cycle is edge-simple and the number of its edges is odd. To formulate the conditions on $x,\Kscr$, we need more notation. As before, for a vertex $v\in V$, we denote by $x_v$ the restriction of $x$ to the set $E_v$ of edges incident to $v$. For $v\in V$, we write:
  \begin{itemize}
\item[$\bullet$] $\Kscr_v$ for the set of cycles in $\Kscr$ passing $v$;
\item[$\bullet$] $\pi_v(K)$ for the set of pairs $(e,e')$ of consecutive edges in a cycle $K\in\Kscr_v$ sharing the vertex $v$;
\item[$\bullet$] $\Kin_v$ and $\Kout_v$ for the sets of edges in $K\in\Kscr_v$ entering and leasing $v$, respectively;
\item[$\bullet$]  $\onebf_v^e$ for the incidence vector of an edge $e\in E_v$, taking value 1 on $e$, and 0 on the other edges in $E_v$;
\item[$\bullet$] $\Deltain_v(K)$ for $\sum(\onebf_v^e\colon e\in\Kin_v)$, and $\Deltaout_v(K)$ for $\sum(\onebf_v^e\colon e\in\Kout_v)$, where $K\in\Kscr_v$;
\item[$\bullet$] $\xin_v$ for $x_v+\sum(\Deltain_v(K)\colon K\in\Kscr_v)$, and $\xout_v$ for $x_v+\sum(\Deltaout_v(K)\colon K\in\Kscr_v)$.
 %
%\item[$\bullet$] 
 \end{itemize}
 
In particular, $\xin_v=\xout_v=x_v$ if $\Kscr_v=\emptyset$.
\medskip

\noindent\textbf{Definition.} We say that $(x,\Kscr)$ is a \emph{half-partnership} if the following conditions hold:
\begin{description}
\item[\rm(C1)] any $v\in V$ and $K\in \Kscr_v$ satisfy: (i) $C_v(\xin_v)=\xin_v$; (ii) $C_v(\xout_v)=\xout_v$; and (iii) $C_v(\xout_v+\Deltain_v(K))=\xout_v+\Deltain_v(K)-\Deltaout_v(K)$;
\item[\rm(C2)] any $v\in V$, \,$K\in\Kscr_v$ and $(e,e')\in\pi_v(K)$ satisfy $C_v(\xout_v+\onebf_v^e)=\xout_v+\onebf_v^e-\onebf_v^{e'}$.
\end{description}
A half-partnership $(x,\Kscr)$ is called \emph{stable} if 
 \begin{description}
\item[\rm(C3)] for any edge $e=uv\in E$ with $x(e)<b(e)$, (i) at least one equality among $C_u(\xin_u+\onebf_u^e)=\xin_u$ and $C_v(\xout_v+\onebf_v^e)=\xout_v$ is valid, and similarly, (ii)~at least one inequality among $C_v(\xin_v+\onebf_v^e)=\xin_v$ and $C_u(\xout_u+\onebf_u^e)=\xout_u$ is valid.
  \end{description}

\noindent(Cf. conditions~1--3 in~\cite[Sec.~2]{flein} for the boolean case.) In particular, if $\Kscr=\emptyset$, then $x$ becomes acceptable and stable, thus giving a stable partnership (since~(C3) implies that no edge $e$ with $x(e)<b(e)$ is blocking).

\begin{theorem} \label{tm:nonbipart}
For $G,b,C$ as above, a stable half-partnership $(x,\Kscr)$ always exists. Moreover, if $(x,\Kscr)$ and $(x',\Kscr')$ are two stable half-partnerships, then $\Kscr=\Kscr'$.
  \end{theorem}
   \begin{proof}
~We show the existence of a stable half-partnership $(x,\Kscr)$ for $G,b,C$ by deriving it from the stable vector $\tilde x$ for $G^\symm,b^\symm,C^\symm$ determined by a closed quasi-balanced function $\lambda$ for $(\hat\Rscr,\lessdot)$ that is lower for \emph{all} rotations $R\in\frakO$. % ( ?? Also for an edge $e\in E^\symm$, we denote its image in $E$ by $\omega(e)$; then $\omega(e)=omega(\sigma(e))$.)

So let $\tilde x:=\phi^{-1}(\lambda)$ for such a $\lambda$. Then $x\in \Zset_+^E$ is defined as follows: for $e=uv\in E$, 
 \begin{numitem1} \label{eq:xtilde-x}
 \begin{itemize}
 \item[(i)]
$x(e):=\tilde x(u^0v^1)=\tilde x(v^0u^1)$ if $u^0v^1$ (as well as $v^0u^1$) is in $E^\symm-E^\symm(\frakO)$; and 
 \item[(ii)] $x(e):=\tilde x(u^0v^1)$ if $u^0v^1\in R^+$ for some $R\in\frakO$ (whence $v^0u^1\in R^-$ and $\tilde x(v^0u^1)=\tilde x(u^0v^1)+1=x(e)+1$; cf.~\refeq{lower}).
  \end{itemize}
  \end{numitem1}
  
The set of cycles $\Kscr$ is determined by $\frakO$ in a natural  way. Namely, for $R\in\frakO$, let $(\tilde e_1,\tilde e_2,\ldots, \tilde e_k)$ be the sequence of positive edges in $R$ such that each $i$, $\tilde e_{i+1}$ is symmetric to the (negative) edge of $R$ next to $\tilde e_i$ (taking indices modulo $k$); this sequence is defined up to cyclically shifting. Then the sequence $(e_1,e_1,\ldots,e_k)$ of their images in $G$ forms an odd cycle, denoted as $\nu(R)$. We define $\Kscr$ to be the set $\{\nu(R)\colon R\in\frakO\}$ and assert that $(x,\Kscr)$ satisfies conditions~(C1)--(C3).

To show (C1)--(C2), consider $v\in V$ and denote by $\frakO_v$ the set of rotations $R\in\frakO$ passing the vertices $v^0,v^1$; then $\frakO_v=\{\nu^{-1}(K)\colon K\in\Kscr_v\}$. Let $\omega_v^0$ and $\omega_v^1$ be the natural bijections $E_v\to E^\symm_{v^0}$ and $E_v\to E^\symm_{v^1}$, respectively (i.e. $\omega_v^i$ maps an edge $uv\in E_v$ to $u^{1-i}v^i$). Then for $e\in E_v$, we have $x(e)=\tilde x(\omega_v^1(e))=\tilde x(\omega^0_v(e))-1$ if $\omega^1_v(e)\in R^+$ (and $\omega_v^0(e)\in R^-$), where $R\in\frakO_v$. Also for $K\in\Kscr_v$ and $R=\nu^{-1}(K)$, ~$\omega_v^1$ maps $\Kin_v$ to $R^+_{v^1}:=R^+\cap E^\symm_{v^1}$, and $\Kout_v$ to $R^-_{v^1}:=R^-\cap E^\symm_{v^1}$; equivalently, $\omega_v^0$ maps $\Kout_v$ to $R^+_{v^0}$, and $\Kin_v$ to $R^-_{v^0}$. This implies that
  \begin{numitem1} \label{eq:xin-xout_maps}
under $\omega^1_v$, the vector $x_v$ is transferred into $\tilde x_{v^1}-\chi_{v^1}^-$; ~$\xout_v$ into $\tilde x_{v^1}$; and $\xin_v$ into $\tilde x_{v^1}+\chi_{v^1}^+ -\chi_{v^1}^-$, where $\chi_{v^1}^+$ stands for $\sum(\chi_{v^1}^{R^+}\colon R\in\frakO_v)$, and similarly for $\chi_{v^1}^-$; 
symmetrically, under $\omega^0_v$, the vector $x_v$ is transferred into $\tilde x_{v^0}-\chi_{v^0}^-$; ~$\xout_v$ into $\tilde x_{v^0}+\chi_{v^0}^+ -\chi_{v^0}^-$; and $\xin_v$ into $\tilde x_{v^0}$.
  \end{numitem1}
  
Then the acceptability relation $C^\symm_{v^1}(\tilde x_{v^1})=\tilde x_{v^1}$ (concerning $\tilde x$) implies $C_v(\xout_v)=\xout_v$; and $C^\symm_{v^1}(\tilde x_{v^1}+\chi_{v^1}^+ -\chi_{v^1}^-)=\tilde x_{v^1}+\chi_{v^1}^+-\chi_{v^1}^-$ (concerning the stable vector $\tilde x+\sum(\chi^R\colon R\in\Oscr)$) implies  $C_v(\xin_v)=\xin_v$; so both $\xout_v$ and $\xin_v$ are acceptable. And to see relation~(iii) in~(C1), consider $K\in\Kscr_v$  and apply~\refeq{CfR+} to $\tilde x_{v^1}$ and $R:=\nu^{-1}(K)$. This gives $C^\symm_{v^1}(\tilde x_{v^1}+\chi_{v^1}^{R^+})=\tilde x_{v^1}+\chi_{v^1}^{R^+}-\chi_{v^1}^{R^-}$, which is transferred into $C_v(\xout_v+\Deltain_v(K))=\xout_v+\Deltain_v(K)-\Deltaout_v(K)$, as required. So~(C1) is valid.

In its turn, to obtain~(C2) with $K\in\Kscr_v$ and $(e,e')\in\pi_v(K)$, consider the consecutive edges $\tilde e:=\omega^1_v(e)$ and $\tilde e':=\omega^1_v(e')$ in the rotation $R:=\nu^{-1}(K)$. Then $\tilde e\in R^+$ and applying~\refeq{e-ep} to $\tilde x_{v^1}$ and $\tilde e$, we have $C^\symm_{v^1}(\tilde x_{v^1}+\onebf^{\tilde e}_{v^1})=\tilde x_{v^1}+\onebf^{\tilde e}_{v^1}-\onebf^{\tilde e'}_{v^1}$. Since $\tilde x_{v^1}$ corresponds to $\xout_v$, this equality is transferred into $C_v(\xout_v+\onebf_v^e)=\xout_v+\onebf_v^e-\onebf_v^{e'}$, yielding~(C2).

The last condition~(C3) follows from the stability of $\tilde x$. More precisely, for $e=uv\in E$ with $x(e)<b(e)$, consider its image $\tilde e=u^0v^1$ in $G^\symm$. Since $\tilde x$ is stable, in case $\tilde x(\tilde e)<b^\symm(\tilde e)$, at least one of the following takes place: (a) $C^\symm_{u^0}(\tilde x_{u^0}+\onebf_{u^0}^{\tilde e})=\tilde x_{u^0}$, or (b) $C^\symm_{v^1}(\tilde x_{v^1}+\onebf_{v^1}^{\tilde e})=\tilde x_{v^1}$. By~\refeq{xin-xout_maps}, the vector $\tilde x_{u^0}$ corresponds to $\xin_u$ (regarding $u^0$ in place of $v^0$), and $\tilde x_{v^1}$ corresponds to $\xout_v$. Then $\tilde x(\tilde e)<b^\symm(\tilde e)$, and the above relations imply that at least one of the equalities $C_u(\xin_u+\onebf_u^e)=\xin_u$ or $C_v(\xout_v+\onebf_v^e)=\xout_v$ is valid, yielding~(C3)(i). And~(C3)(ii) is shown similarly, by considering the edge $\tilde e=v^0u^1$.

This gives the first assertion in the theorem. The second assertion in it (concerning the invariance of $\Kscr$) is shown by ``conversing'' the above reasonings.

More precisely, let $(x,\Kscr)$ satisfy (C1)--(C3). We define the corresponding vector $\tilde x\in E^\symm$ as follows: for $e=uv\in E$,
  \begin{numitem1} \label{eq:x-tilde_x}
\begin{itemize}
\item[(i)] if $e$ is outside $\Kscr$, then $\tilde x(u^0v^1)=\tilde x(v^0u^1):=x(e)$; and
\item[(ii)] if $e$ belongs to a cycle $K\in\Kscr$ and if $u,e,v$ follow in this order in $K$, then $\tilde x(u^0v^1):=x(e)$ and $\tilde x(u^1v^0):=x(e)+1$.
\end{itemize}
 \end{numitem1}

Then for each $v\in V$, the map $\omega_v^1$ transfers the vector $\xout$ into into $\tilde x_{v^1}$, while $\omega_v^0$ transfers $\xin_v$ into $\tilde x_{v^0}$ (cf.~\refeq{xin-xout_maps}). As a consequence, condition~(C3)(i) implies that for any edge $\tilde e=u^0v^1\in E^\symm$ with $\tilde x(\tilde e)< b^\symm(\tilde e)$, at least one of $C^\symm_{u^0}(\tilde x_{u^0}+\onebf_{u^0}^{\tilde e})=\tilde x_{u^0}$ and $C^\symm_{v^1}(\tilde x_{v^1}+\onebf_{v^1}^{\tilde e})=\tilde x_{v^1}$ is valid. Therefore, $\tilde x$ is stable.

It follows that the symmetric function $\tilde x^\ast$ is stable as well. Clearly $\tilde x$ and $\tilde x^\ast$ coincide on the edges whose images in $G$ are not in $\Kscr$. We assert that $\tilde x^\ast-\tilde x$ is represented as the sum of incidence vectors $\chi^R$ of singular rotations $R$ with odd weights $\tau_R$.

To show this, consider a cycle $K\in\Kscr$; let $K=(v_0,e_1,v_1,e_2,\ldots,e_k,v_k=v_0)$ (with $k$ odd). Then
 \begin{numitem1} \label{eq:K-R}
\begin{itemize}
 \item[(i)] 
$K$ induces in $G^\symm$ the $2k$ cycle $R:=\mu(K)$ with edges $e_i^1:=v_{i-1}^0v_i^1$ and $e_i^0:=v_{i-1}^1v_i^0$ ($i=1,\ldots,k$) that follow in $R$ in the sequence $e_1^1,e_2^0,\ldots, e_{k-1}^0,e_k^1,e_1^0,e_2^1,\ldots,e_k^0$; and (using \refeq{x-tilde_x}(ii))
  \item[(ii)]
 for $i=1,\ldots,k$, $\tilde x(e_i^1)=\tilde x^\ast(e_i^0)=x(e_i)$ and $\tilde x(e_i^0)=\tilde x^\ast(e_i^1)=x(e_i)+1$.
 \end{itemize}
  \end{numitem1}

The edges in $R^+:=\{e_1^1,\ldots,e_k^1\}$ and $R^-:=\{e_1^0,\ldots,e_k^0\}$ are regarded as \emph{positive} and \emph{negative}, respectively, and we denote by $\chi^R$ the corresponding $0,\pm 1$ incidence vector for $R$ in $\Zset_+^{E^\symm}$ (like for usual rotations). Also we write $\tilde \frakO$ for the set of cycles $R=\mu(K)$ over $K\in\Kscr$. Then~\refeq{K-R} implies that
  \begin{equation} \label{eq:tilde_x-xast}
  \tilde x^\ast =\tilde x+\sum(\chi^R\colon R\in\tilde \frakO).
  \end{equation}
  
Now consider a cycle $R\in\tilde\frakO$ and a vertex $v^1$ in it (contained in $F$), and denote by $\chi_{v^1}^{R^+}$ and $\chi_{v^1}^{R^-}$ the restrictions to $E^\symm_{v^1}$ of the $0,1$ incidence vectors $\chi^{R^+}$ and $\chi^{R^-}$, respectively. Let $\tilde x':=\tilde x+\chi^R$. One can see that relation~(C1)(iii) with the vertex $v$ and cycle $K:=\mu^{-1}(R)$ gives rise to the equalities
   $$
  C^\symm_{v^1}(\tilde x_{v^1}+\chi_{v^1}^{R^+}) =\tilde x_{v^1}+\chi_{v^1}^{R^+}-\chi_{v^1}^{R^-}=\tilde x'_{v^1}.
  $$
This implies $C_{v^1}^\symm(\tilde x_{v^1}\vee\tilde x'_{v^1})=\tilde x'_{v^1}$, which means that $\tilde x_{v^1}\prec_{v^1} \tilde x'_{v^1}$. Combining the latter relations over the vertices in $F$ covered by $\tilde\frakO$ and cycles in $\tilde\frakO$ and using~\refeq{tilde_x-xast}, we obtain that $\tilde x\prec_F\tilde x^\ast$. Therefore, $\tilde x^\ast$ is obtained from $\tilde x$ by applying a series of increasing rotations, and we now explain that these rotations are just the cycles $R$ in $\tilde \frakO$.

To see this, consider a cycle $K=(v_0,e_1,v_1,\ldots,e_k,v_k)$ in $\Kscr$ and its image $R=\mu(K)$. For each $i$, the relation in~(C2) with $v=v_i$, $e=e_i$ and $e'=e_{i+1}$ is transferred into 
  $$
  C^\symm_{v^1}(\tilde x_{v^1}(e_i^1)+\onebf_{v^1}^{e^1_i})
   =\tilde x_{v^1}(e_i^1)+\onebf_{v^1}^{e^1_i}-\onebf_{v^1}^{e_{i+1}^0}
   $$
(since the map $\omega^1_{v}$ turns $\xout_{v}$ into $\tilde x_{v^1}$), yielding the ``link'' $(e_i^1,e_{i+1}^0)$ in $R$ (cf.~\refeq{e-ep}). Symmetrically, considering $\tilde x':=\tilde x+\chi^R$, appealing to the map $\omega_{v}^0$, and using the equalities $\tilde x'(e_i^0)=\tilde x(e_i^1)=x(e_i)$ and $\tilde x'(e_{i+1}^1)=\tilde x(e_{i+1}^0)=x(e_{i+1})+1$, we obtain
   $$
  C^\symm_{v^0}(\tilde x'_{v^0}(e_i^0)+\onebf_{v^0}^{e^0_i})
   =\tilde x'_{v^0}(e_i^0)+\onebf_{v^0}^{e^0_i}-\onebf_{v^0}^{e_{i+1}^1},
   $$
which gives the ``link'' $(e_i^0,e_{i+1}^1)$ in $R$ (cf. explanations in Remark~1). These ``links'' just determine the cycle $R$. So $R$ is an increasing singular rotation applicable to $\tilde x$.

Finally, considering the closed functions $\lambda:=\phi(\tilde x)$ and $\eps:=\phi(\tilde x^\ast)$, we have $\eps(R)=\lambda(R)+1$ (in view of~\refeq{K-R}(ii)). Then the equality $\lambda(R)+\eps(R)=\tau_R$ in Lemma~\ref{lm:x-xstar} (with $R=R^\ast$) implies that $\tau_R$ is odd.

Thus, we can conclude that $\tilde\frakO$ is exactly the set of singular rotations with odd maximal weights, and that $\tilde \frakO$ determines the set $\Kscr$ in a stable half-partnership $(x,\Kscr)$.

This completes the proof of the theorem. 
 \end{proof}

%----------------------- Sec.6

\section{Concluding remarks} \label{sec:concl}

In conclusion of this paper, we briefly outline how the above ``double copying'' approach can be applied to two more versions of stability problem on non-bipartite graphs. 
\medskip

\noindent\textbf{I.}
Section~7 of~\cite{karz} is devoted to one special case of the integer Alkan--Gale's stability problem (IAGP) on bipartite graphs. It is specified by imposing an additional axiom on the choice functions, the so-called gapless condition. A similar specification can be done in the non-bipartite version as well. More precisely, consider the stable partnership problem (SPPIC) for a non-bipartite graph $G=(V,E)$ with integer edge capacities $b$ and integer choice functions $C_v$, $v\in V$, in which, in addition to axioms SUB and MON (see Sect.~\SEC{defin}), the following \emph{gapless condition} is imposed on each $C_v$:
  \begin{itemize}
  \item[(GL)] if acceptable vectors $z_1,z_2,z_3\in\Ascr_v$ and edges $a,c_1,c_2,c_3\in E_v$ satisfy the relations: (i) $z_1\prec_v z_2\prec_v z_3$, (ii) $C_v(z_i+\onebf_v^a)=z_i+\onebf_v^a-\onebf_v^{c_i}$ for $i=1,2,3$, and (iii) $c_1=c_3$, then the equality $c_1=c_2$ is valid as well.
   \end{itemize}
In particular, this holds for the stable allocation problem by Baiou and Balinsky.

As is shown in~\cite[Sec.~7]{karz}, the stability problem IAGP (when $G$ is bipartite) becomes solvable in weakly polynomial time. More precisely, one shows that (see Theorem~7.2 in~\cite{karz}): all rotations in $\hat \Rscr$ are different, $|\Rscr|=|\Pi|\le 4|E|^2|F| |W|$, and the poset $(\Rscr,\tau,\lessdot)$ can be constructed in time $\log\bmax$ times a polynomial in $|V|$ (estimating the number of oracle calls and standard operations).

Relying on these results, the reduction of SPPIC to IAGP as described in Sects.~\SEC{symmetr},\SEC{nonbipartite}, leads to the following
  \begin{corollary} \label{cor:gap_cond}
In the non-bipartite SPPIC, when condition (GL)  is imposed on all choice functions, a stable half-partnership $(x,\Kscr)$ for $(G,b,C)$ can be found in weakly polynomial time, namely, in time $\log\bmax$ times a polynomial in $|V|$.
 \end{corollary}

\noindent\textbf{II.}
Finally, we consider a non-bipartite version of the problem on stable assignments generated by choice functions of the so-called \emph{mixed type} (briefly \emph{SMP}) studied in~\cite{karz2}.  It deals with a graph $G=(V,E)$ with edge capacities $b\in\Rset_+^E$ and \emph{quotas} $q(v)\in\Rset_+$ at the vertices $v\in V$, the upper bounds on the total assignments on $E_v$. There are two equivalent ways to state the problem. 

The first way (a ``combinatorial'' setting) arises when, in the stable allocation problem with $G,b,q$, we relax preferences at the vertices by replacing strong linear orders on the sets $E_v$, $v\in V$, by \emph{weak} ones. The latter can be described by an ordered partition $\pi_1>_v \pi_2>_v\cdots >_v \pi_k$ of $E_v$, where for edges $e\in\pi_i$ and $e'\in\pi_j$, ~$e$ is preferred to $e'$ if $i<j$, and $e,e'$ are regarded as equivalent if $i=j$. In the second way (an ``analytic'' setting), the above ordered partition is associated with the appropriate choice function $C_v$ that handles the vectors $z\in\Rset_+^{E_v}$ not exceeding the capacities $b$ and obeying:
  \begin{numitem1} \label{eq:z-q}
  \begin{itemize}
\item[(a)] if $|z|\le q(v)$ then $C_v(z)=z$;
\item[(b)] if $|z|> q(v)$, then in the ordered partition as above we take the group $\pi_i$ such that $z(\pi_1\cup\cdots\cup \pi_{i-1})< q(v)\le z(\pi_1\cup\cdots\cup \pi_{i})$, and define $C_v(z)(e)$ to be: $z(e)$ if $e\in \pi_j$ for $j<i$; ~$0$ if $e\in \pi_j$ for $j>i$; and $\min\{r,z(e)\}$ if $e\in\pi_i$,
where the value $r$ (\emph{the cutting level} in $\pi_i$) is computed from the equality
  $$
  z(\pi_1\cup\cdots\cup \pi_{i-1})+\sum(\min\{r,z(e)\} \colon e\in \pi_i)=q(v)
  $$
(implying $|C_v(z)|=q(v)$).
  \end{itemize}
  \end{numitem1} 

When the input graph $G=(V,E)$ is bipartite, with vertex parts $W$ and $F$, SMP falls in the framework of Alkan--Gale's model, and therefore (cf.~\refeq{AG}) the set $\Sscr$ of stable assignments for $(G,b,q,C)$ is nonempty and forms a distributive lattice $\Lscr=(\Sscr,\prec_F)$, with minimal element $\xmin$ and maximal element $\xmax$. An extensive study of this lattice conducted in~\cite{karz2} refined the structure and properties of rotations arising in SMP, namely:
  \begin{numitem1} \label{eq:rot_in_SMP}
  \begin{itemize}
\item[(i)] the set $\Rscr$ of rotations related to $\Lscr$ is finite and has size $|\Rscr|$ at most $2|E|$;
\item[(ii)] each rotation $\rho\in\Rscr$ is a $\Zset$-\emph{circulation}, which means that $\rho\in\Zset^E$, and for each $v\in V$, the sum $\sum(\rho(e)\colon e\in E_v)$ is zero; each entry of $\rho$ has the encoding size (in binary notation) bounded by a (fixed) polynomial in $|V|$;
 \item[(iii)] for each stable assignment $x\in\Sscr-\{\xmax\}$, there exist a rotation $\rho\in\Rscr$ and a real $\lambda>0$ such that the vector $x':=x+\lambda\rho$ is stable and $x\prec_F x'$; in this case, we say that $\rho$ is \emph{applicable} to $x$;
  \item[(iv)] for each $x\in \Sscr$, the rotations $\rho$ applicable to $x$ have pairwise disjoint supports $\{e\in E\colon \rho(e)\ne 0\}$; the set $\Rscr(x)$ of these rotations is constructed in strongly polynomial time;
  \item[(v)] for each $x\in \Sscr$, there exists a sequence (``non-excessive route'') $\Tscr=(x_0,x_1,\ldots,x_k)$ of stable assignments with $x_0=\xmin$ and $x_k=x$ such that: $x_i=x_{i-1}+\lambda_i\rho_i$ for each $i=1,\ldots,k$, where $\lambda_i>0$ and $\rho_i$ is a rotation applicable to $x_{i-1}$, and all rotations $\rho_i$ are different; moreover, the set of pairs $\{(\rho_i,\lambda_i)\colon i=1,\ldots,k\}$ is invariant of such a $\Tscr$.
 \end{itemize}
  \end{numitem1}

(Note that when $b$ and $q$ are integer-valued, an integer stable assignment need not exist; yet, in this case there exists a stable assignment in which each entry is a rational number whose denominator has encoding size bounded by a polynomial in $|V|$.)

Also it was shown in~\cite{karz2} that
  \begin{numitem1} \label{eq:biject_SMP}
there exist a partial order $\lessdot$ on $\Rscr$ and a weight function $\tau:\Rscr\to \Rset_{>0}$ such that:
   \begin{itemize}
\item[(a)] there is a map $\phi$ from $\Sscr$ to a set of nonnegative real functions $\lambda$ on $\Rscr$ satisfying $x=\xmin+\sum(\lambda\rho\colon \rho\in\Rscr)$ for all $x\in\Sscr$ and $\lambda:=\phi(x)$, and this $\phi$ establishes an isomorphism between $(\Sscr,\prec_F)$ and the lattice of closed functions for $(\Rscr,\tau,\lessdot)$, where $\lambda:\Rscr\to\Rset_+$ is called \emph{closed} if  $\lambda(\rho)\le\tau(\rho)$ for all $ \rho\in\Rscr$, and for any $\rho,\rho'\in\Rscr$ with $\rho\lessdot\rho'$ and $\lambda(\rho')>0$, there holds $\lambda(\rho)=\tau(\rho)$;
   \item[(b)] the weighted rotational poset $(\Rscr,\tau,\lessdot)$ can be constructed in strongly polynomial time.
     \end{itemize}
     \end{numitem1}
     
Relying on this background, we now consider a non-bipartite $G=(V,E)$ and reduce SMP with $G,b,q,C$ to the corresponding bipartite version of SMP with symmetric $G^\symm=(V^\symm=V^0\cup V^1,E^\symm), b^\symm,q^\symm,C^\symm$, where $q^\symm(u^0v^1)=q^\symm(v^0u^1):=q(uv)$ for $uv\in E$ (similarly to the reduction of SPPIC to IAGP in Sect.~\SEC{defin}). Then there is a natural bijection between the domain $\Bscr=\{x\in \Rset_+^E\colon x\le b,\;|x_v|\le q(v)\,\forall \,v\in V\}$ and  the set of symmetric nonnegative functions on $E^\symm$ bounded by $b^\symm$ and $q^\symm$; namely, the bijection $x\stackrel{\beta}\longmapsto y$ such that $x(uv)=y(u^0v^1)=y(v^0u^1)$ for all $uv\in E$. Moreover (cf.~\refeq{biject}),
  \begin{numitem1} \label{eq:stable-stable}
if $x\in\Bscr$ is stable for $(G,b,q,C)$, then the symmetric assignment $\beta(x)$ is stable for $(G^\symm,b^\symm,q^\symm,C^\symm)$, and vice versa.
  \end{numitem1}
     
Therefore, to solve SMP with $(G,b,q,C)$ it suffices to find a symmetric stable solution for $(G^\symm,b^\symm,q^\symm,C^\symm)$. Note that here we deal with real-valued assignments (in contrast to integer ones for SPPIC in the main stream of the paper), and the machinery elaborated in Sect.~\SEC{symmetr} enables us to fulfill this task relatively easily. 

More precisely, denote by $\sigma$ the symmetry operator for objects in $G^\symm$, denote by $x^\ast$ the vector symmetric to $x\in\Sscr$ (i.e. $x^\ast(e)=x(\sigma(e))$ for $e\in E^\symm$), and for a rotation $\rho\in\Rscr$, denote by $\rho^\ast$ the functions on $E^\symm$ defined by $\rho^\ast(e):=-\rho(\sigma(e))$, $e\in E^\symm$ (where $\rho^\ast$ is regarded as symmetric to $\rho$). When $\rho^\ast=\rho$, we say that $\rho$ is self-symmetric, or \emph{singular}.

Arguing as in Sect.~\SEC{symmetr} (or as in~\cite[Sec.~5]{DM}), one can obtain the following analogs of~\refeq{antisym}, Lemma~\ref{lm:x-xstar} and Corollary~\ref{cor:symm_closed} (a verification of these and further assertions is left to the reader as an exercise):
   \begin{numitem1} \label{eq:rho-x}
   \begin{itemize}
\item[(i)] for each $\rho\in\Rscr$, $\rho^\ast$ belongs to $\Rscr$ as well, $\tau(\rho)=\tau(\rho^\ast)$, and the relation $\lessdot$ is antisymmetric: if $\rho,\rho'\in\Rscr$ and $\rho\lessdot \rho'$, then $\rho^\ast\gtrdot \rho'^\ast$;
  \item[(ii)] for any $x\in\Sscr$, the closed functions $\lambda:=\phi(x)$ and $\eps:=\phi(x^\ast)$ for $(\Rscr,\tau,\lessdot)$ satisfy the equality $\lambda(\rho)+\eps(\rho^\ast)=\tau(\rho)$ for each $\rho\in\Rscr$;
   \item[(iii)] the set $\Sscr^\symm$ of symmetric stable vectors is nonempty, and the map $\phi$ establishes a bijection between $\Sscr^\symm$ and the set of closed functions $\lambda$ such that $\lambda(\rho)+\lambda(\rho^\ast)=\tau(\rho)$ for all $\rho\in\Rscr$.
     \end{itemize}
     \end{numitem1}
     
Now to find a symmetric stable $x\in\Sscr^\symm$, we use an analog of algorithm QB from Sect.~\SEC{symmetr}, up to some modification. More precisely, we construct, step by step, a route $\Tscr$ starting with $x:=\xmin$. At each iteration, for the current $x$, we scan the set $\Rscr(x)$ of rotations applicable to $x$ to seek for a rotation $\rho$ in it such that its symmetric $\rho^\ast$ is not used in the current $\Tscr$. If such a $\rho$ is found, then we assign $\lambda(\rho):=\tau(\rho)$ if $\rho$ is non-singular, and $\lambda(\rho):=\tau(\rho)/2$ otherwise, after which we update $x:=x+\lambda(\rho)\rho$ and finish the iteration. The algorithm terminates with the current $x$ when for each $\rho\in\Rscr(x)$, $\rho^\ast$ is already used in $\Tscr$.

Arguing as in the proof of Lemma~\ref{lm:quasi-bal} (even simpler), one can show that $x$ and $\lambda$ obtained upon termination of the above algorithm satisfy the equalities as in~\refeq{rho-x}(iii), which implies $x\in\Sscr^\symm$. Summing up the above reasonings, we conclude with
  \begin{theorem} 
In the non-bipartite version of SMP (viz. the non-bipartite relaxed stable allocation problem with weak linear orders at the vertices), a stable assignment exists and can be found in strongly polynomial time.
  \end{theorem}


\begin{thebibliography}{99}
\small

\bibitem{AG}
A.~Alkan and D.~Gale. Stable schedule matching under revealed preference.
\textsl{J.~Economic Theory} \textbf{112} (2003) 289--306.
 %
\vspace{-7pt}
 \bibitem{BB}
M.~Baiou and M.~Balinski. Erratum: the stable allocation (or ordinal
transportation) problem. \textsl{Math. Oper. Res.} \textbf{27} (4) (2002)
662--680.
 %
\vspace{-7pt}
  \bibitem{birk} G.~Birkhoff. Rings of sets. \textsl{Duke Mathematical Journal} \textbf{3} (3) (1937) 443--454.
 %
\vspace{-7pt}
 \bibitem{BF}
P. Bir\'o and T. Fleiner. The integral stable allocation problem on graphs. \textsl{Discrete Optimization} \textbf{7} (2010) 64--73.
 %
 %
\vspace{-7pt}
 \bibitem{blair}
C. Blair. The lattice structure of the set of stable matchings with multiple
partners. \textsl{Math. Oper. Res.} \textbf{13} (1988) 619--628.
 %
\vspace{-7pt}
  \bibitem{DM} B.C.~Dean and S.~Munshi, Faster algorithms for stable allocation problems.
\textsl{Algorithmica} \textbf{58} (1) (2010) 59--81.
 %
\vspace{-7pt}
  \bibitem{flein}
T.~Fleiner, The stable roommates problem with choice functions. \textsl{EGRES Technical Report} No. 2007-11, 2007.
 % 
\vspace{-7pt}
 \bibitem{GS}
D.~Gale and L.S.~Shapley. College admissions and the stability of marriage.
\textsl{Amer. Math. Monthly} \textbf{69} (1) (1962) 9--15.
   %
\vspace{-7pt}
  \bibitem{hsu} Y.-C.~Hsueh, A unifying approach to the structures of stable matching problems. \textsl{Computers Math. Appl.} \textbf{22} (6) 1991 13--27.
 %
\vspace{-7pt}
 \bibitem{irv}
R.W.~Irving, An efficient algorithm for the “stable roommates” problem.
\textsl{J. Algorithms} \textbf{6} (1985) 577--595.
 %
\vspace{-7pt}
 \bibitem{IL}
R.W. Irving and P. Leather. The complexity of counting stable marriages. \textsl{SIAM J. Comput.} \textbf{15} (1986) 655--667.
 %
\vspace{-7pt}
\bibitem{ILG}
R.W.~Irving, P.~Leather and D.~Gusfield. An efficient algorithm for the optimal
stable marriage problem. \textsl{J. ACM} \textbf{34} (1987) 532--543.
%
\vspace{-7pt}
 \bibitem{karz2}
A.V. Karzanov. On stable assignments generated by choice functions of mixed
type. \textsl{Discrete Applied Math.} \textbf{358} (2024) 112--135.
 %
\vspace{-7pt}
\bibitem{karz} A.V. Karzanov. A poset representation for stable contracts in a two-sided market generated by integer choice functions. \textsl{arXiv}:2512.05942 [math.CO], 2025.
   %
\vspace{-7pt}
   \bibitem{KC}
A.S. Kelso and V.P. Crawford. Job matching, coalition formation and gross
substitutes. \textsl{Econometrica} \textbf{50} (1982) 1483--1504.
 %
\vspace{-7pt}
  \bibitem{plott}
C.R.~Plott. Path independence, rationality, and social choice.
\textsl{Econometrica} \textbf{41} (6) (1973) 1075--1091.
    %
  %
\vspace{-7pt}
 \bibitem{roth}
A.E. Roth. Stability and polarization of interests in job matching.
\textsl{Econometrica} \textbf{52} (1984) 47--57.
 %
 \vspace{-7pt}
  \bibitem{tan}
J.J.M.~Tan, A necessary and sufficient condition for the existence of a complete stable matching. \textsl{J. of Algorithms} \textbf{12} (1) (1991) 154--178.
   %
\vspace{-7pt}
  \bibitem{TH}
J.J.M. Tan and Y.C. Hsueh, A generalization of the stable matching problem. \textsl{Discrete Appl. Math.} \textbf{59} (1) (1995) 87--102.

%
\end{thebibliography}
\end{document}